\documentclass[reqno, 12pt]{amsart}

\usepackage{mathtools}
\usepackage[T1]{fontenc}
\usepackage{amsmath, amsthm, amssymb}
\usepackage[ansinew]{inputenc}
\usepackage{xcolor}
\usepackage{bbm}
\usepackage{mathrsfs}
\usepackage{cite}
\usepackage{xparse}
\usepackage{fullpage}

\newtheorem{theorem}{Theorem}[section]
\newtheorem{lemma}[theorem]{Lemma}
\newtheorem{prop}[theorem]{Proposition}
\newtheorem{cor}[theorem]{Corollary}

\theoremstyle{definition}
\newtheorem{definition}[theorem]{Definition}

\theoremstyle{remark}
\newtheorem{remark}[theorem]{Remark}

\numberwithin{equation}{section}

\DeclareMathOperator{\rg}{rg}
\DeclareMathOperator{\rank}{rank}

\begin{document}

\title{On blowup of co-rotational wave maps in odd space dimensions}

\author{Athanasios Chatzikaleas}
\address{Rheinische Friedrich - Wilhelms - Universit\"at Bonn, Mathematisches Institut, Endenicher Allee 60, D-53115 Bonn, Germany}
\email{achatzik@math.uni-bonn.de}

\author{Roland Donninger}
\address{Rheinische Friedrich - Wilhelms - Universit\"at Bonn, Mathematisches Institut, Endenicher Allee 60, D-53115 Bonn, Germany}
\address{Universit\"at Wien, Fakult\"at f\"ur Mathematik, Oskar-Morgenstern-Platz 1, A-1090 Vienna, Austria}
\email{donninge@math.uni-bonn.de, roland.donninger@univie.ac.at}
\thanks{Roland Donninger is supported by the Alexander von Humboldt Foundation via
a Sofja Kovalevskaja Award endowed by the German Federal Ministry of Education
and Research. Partial support by the Deutsche Forschungsgemeinschaft 
(DFG), CRC 1060 'The Mathematics of Emergent Effects', is also gratefully acknowledged.}


\author{Irfan Glogi\'c}
\address{Department of Mathematics, The Ohio State University, 231 W 18th Ave,
Columbus, OH, 43220, USA}
\email{glogic.1@osu.edu}


\dedicatory{}

\begin{abstract}
We consider co-rotational wave maps from the $(1+d)$-dimensional Minkowski space into the $d$-sphere for $d\geq 3$ odd. This is an energy-supercritical model which is known to exhibit finite-time blowup via self-similar solutions. Based on a method developed by the second author and Sch\"orkhuber, we prove the asymptotic nonlinear stability of the ``ground-state'' self-similar solution. 
\end{abstract}

\maketitle

\section{Introduction} 

\noindent
Let $(M,g)$ be a Lorentzian spacetime and $(N,h)$ a Riemannian manifold. In this paper, we study wave maps $u:(M,g) \longrightarrow (N,h)$, that is, critical points of the geometric action functional
\begin{align*}
S_{g}[u]:=\frac{1}{2} \int _{M} |d_{g} u|^{2}~d \mu_{g}. 
\end{align*}
Here, 
\begin{align*}
|d_{g}u(x)|^{2} \equiv  |d_{g}u (x)|_{ T^{\star}_{x}M \otimes T_{u(x)} N  }^{2}:= \text{tr}_{g} \left (  u^{\star} \left( h \right) \right) 
\end{align*}
is the trace (with respect to $g$) of the pullback metric on $(M,g)$ via the map $u$. The integral is understood with respect to the standard measure $d\mu _{g}$ on the domain manifold. In local coordinates $(x_{\mu})$ on $(M,g)$, this expression reads
\begin{align*}
S_{g}[u]= \int_{M} g ^{\mu \nu} (\partial _{\mu}  u^{a}) (\partial _{\nu}  u^{b}) h_{a b} \circ u~ d \mu_{g}
\end{align*}
where the Einstein summation convention is used. The Euler-Lagrange equations associated to this functional are
\begin{align} \label{WM}
\Box _{g} u^{a} + g ^{\mu \nu} (\Gamma ^{a}_{b c} \circ u) (\partial _{\mu}  u^{b} ) (\partial _{\nu}  u^{c}) =0
\end{align}
and they constitute a system of semi-linear wave equations. Here, $\Box _{g}$ is the Laplace-Beltrami operator on $(M,g)$ 
\begin{align*}
\Box_{g} := \frac{1}{|g|} \partial _{\mu} (g^{\mu \nu} |g| \partial _{\nu}),\quad |g|:=\sqrt{ \left| \text{det}(g_{\mu \nu}) \right| }
\end{align*}
and $\Gamma ^{a}_{b c}$ are the Christoffel symbols associated to the metric $h$ on the target manifold. Eq.~\eqref{WM} is called the wave maps equation (known in the physics literature as non-linear $\sigma$ model) and is the analog of harmonic maps between Riemannian manifolds in the case where the domain is a Lorentzian manifold instead. For more details, we refer the reader to \cite{Ren08} and \cite{Str97}.

\subsection{Intuition} Recently, the wave maps equation has attracted a lot of interest. On the one hand, the wave maps equation is a rich source for understanding nonlinear geometric equations since it is a nonlinear generalization of the standard wave equation on Minkowski space. In addition, the wave maps equation has a pure geometric interpretation: it generalizes the notion of geodesic curves. Notice that, if $M = (\alpha, \beta)$ is an open interval and $(N,h)$ any curved Riemannian manifold, the wave maps equation is the geodesic equation
\begin{align*}
\frac{d^2 u^{a}}{dt ^2} (t) + (\Gamma ^{a}_{b c} \circ u(t))  \frac {d u^{b} }{dt} (t) \frac{d  u^{c} }{dt} (t)=0. 
\end{align*}
On the other hand, the Cauchy problem for the wave maps system provides an attractive toy-model for more complicated relativistic field equations. Specifically, wave maps contain many features of the more complex Einstein equations but are simple enough to be accessible for rigorous mathematical analysis. 
Further details on the correlation between the wave maps system and the Einstein equations can be found in \cite{Mis78, Mon89, Wei90, Kla97}. 

Being a time evolution equation, the fundamental problem is the Cauchy problem: given specified smooth initial data, does there exist a unique smooth solution to the wave maps equation with this initial data? Furthermore, does the solution exist for all times? On the other hand, if the solution only exists up to some finite time $T$, how does the solution blow up as $t$ approaches $T$? 
The investigation of questions of global existence and formation of singularities for the wave maps equation can give insight into the analogous, but much more difficult, problems in general relativity. 
\subsection{Equivariant wave maps} Now, we turn our attention to the Cauchy problem in the case where the domain is the Minkowski spacetime $(\mathbb{R}^{1+d},g)$ and the target manifold is the sphere $(\mathbb{S}^{d},h)$ for $d \geq 3$. Hence, we pick $g =$diag$(-1,1,\dots,1)$ and $h$ to be the standard metric on the sphere. Furthermore, we choose standard spherical coordinates on Minkowski space and hyper-spherical coordinates on the sphere. The respective metrics are given by
\begin{align*}
g = - dt^2 + dr^2 + r^2 d \omega ^2, \quad h= d \Psi ^2 + \sin^2(\Psi) d \Omega ^2,
\end{align*}
where $d\omega ^2$ and $d \Omega ^2$ are the standard metrics on $\mathbb{S}^{d-1}$.  Moreover, a map $u:(\mathbb{R}^{1+d},g) \longrightarrow (\mathbb{S}^{d},h)$ can be written as
\begin{align*}
u (t,r,\omega) = \big( \Psi  (t,r,\omega ), \Omega (t,r,\omega) \big).
\end{align*}
We restrict our attention to the special subclass known as 1-equivariant or co-rotational, that is
\begin{align*}
\Psi  (t,r,\omega ) \equiv \psi (t,r),~~~ \Omega (t,r, \omega ) = \omega. 
\end{align*}
 Under this ansatz, the wave maps system for functions $u:(\mathbb{R}^{1+d},g) \longrightarrow (\mathbb{S}^{d},h)$ reduces to the single semi-linear wave equation 
 \begin{equation}
 \label{eq:main} \psi_{tt}-\psi_{rr}-\frac{d-1}{r}\psi_r+\frac{d-1}{2}\frac{\sin(2\psi)}{r^2}=0. 
 \end{equation}
 By finite speed of propagation and radial symmetry it is natural to
study this equation in backward light-cones with vertex $(T,0)$, that is
\begin{align*}
C_{T} :=\left \{  (t,r) : 0<t<T,~0 \leq r \leq T-t \right \}
\end{align*}
where $T>0$.
Consequently, we consider the Cauchy problem \\
\begin{align} \label{cauchy}
     \begin{cases}
       \psi _{tt} (t,r) - \Delta ^{  \text{rad}  }_{r,d} \psi (t,r)  = - \frac{d-1}{2} \frac{ \sin ( 2 \psi (t,r) ) }{r^2}, &\quad \text{in } C_{T} \\
       \psi (0,r)= f(r),~~~ \psi _{t} (0,r)= g(r), &\quad \text{on } \{ t=0 \} \times [0,+\infty)  \ 
     \end{cases}
\end{align}	
where $\Delta ^{  \text{rad}  }_{r,d} $ stands for the radial Laplacian 
\begin{align*}
\Delta ^{  \text{rad}  }_{r,d} \psi (t,r) :=  \psi _{rr} (t,r) + \frac{d-1}{r} \psi _{r} (t,r).
\end{align*}
To ensure regularity of solutions, equations $\eqref{cauchy}$ must be supplemented by the boundary condition
\begin{align} \label{reg}
\psi (t,0)=0,\quad  \text{~for all~} t \in (0,T).
\end{align}

\subsection{Self-similar solutions} 
A basic question for the Cauchy problem $\eqref{cauchy}$ is whether solutions starting from smooth initial data 
\begin{align*}
(f,g)=\left( \psi (0, \cdot), \partial _{t} \psi (0, \cdot) \right)
\end{align*}
can become singular in the future. Note that Eq.~\eqref{eq:main} has the conserved energy
\begin{align*}
E[\psi]:=\int_{0}^{\infty} \left( \psi _{t}^2 + \psi _{r}^2 + (d-1)\frac{\sin ^2(\psi) }{r^2} \right) r^2 dr .
\end{align*}
However, the energy cannot be used to control the evolution since Eq.~\eqref{cauchy} is not well-posed at energy regularity, cf.~\cite{ShaTah94}. 
Indeed, Eq.~\eqref{eq:main} is invariant under dilations
\begin{align} \label{dilation}
\psi _{\lambda} (t,r):=\psi \left( \frac{t}{\lambda},\frac{r}{\lambda} \right),~ \lambda >0
\end{align}
and the critical Sobolev space for the pair $(\psi(t,\cdot),\partial_t \psi(t,\cdot))$ is $\dot H^{\frac{d}{2}}\times \dot H^{\frac{d}{2}-1}$.
Consequently, Eq.~\eqref{eq:main} is energy-supercritical for $d\geq 3$.

In fact, due to the scaling $\eqref{dilation}$ and the supercritical character it is natural to expect self-similar solutions and indeed, it is well known that there exist smooth initial data which lead to solutions that blowup in finite time in a self-similar fashion. Specifically, Eq.~\eqref{eq:main} admits the self-similar solution
\begin{align*}
\psi ^{T} (t,r) := f_{0} \Big ( \frac{r}{T-t } \Big) = 2 \arctan \Bigg( \frac{r}{\sqrt{d-2}(T-t)  } \Bigg),\qquad T>0.
\end{align*}
This example is due to Shatah \cite{Sha88}, Turok-Spergel \cite{TurSpe90} for $d=3$, and  Bizo\'n-Biernat \cite{BizBie15} for $d \geq 4$ and provides an explicit example for singularity formation from smooth initial data. Indeed, the self-similar solution $\psi^T$ is perfectly smooth for all $0<t < T$ but breaks down at $t = T$ in the sense that 
\begin{align*}
\partial_r\psi ^{T} (t,r)|_{r=0} 
\simeq \frac{1}{T-t} \longrightarrow + \infty,~~~\text{as}~ t \longrightarrow T^{-}.
\end{align*}
We note in passing that for $d\in \{3,4,5,6\}$, $\psi^T$ is just one member of a countable family of self-similar solutions, see \cite{Biz00, BieBizMal16}. 

\subsection{The main result}
By finite speed of propagation one can use $\psi ^{T}$ to construct smooth, compactly supported initial data which lead to a solution that blows up as $t \longrightarrow T$. Our main theorem is concerned with the asymptotic nonlinear stability of $\psi^T$. In other words, we prove the existence of an open set of radial data which lead to blowup via $\psi ^{T}$. In this sense, the blowup described by $\psi ^{T}$ is stable. To state our main result, we will need the notion of the blowup time at the origin.
From now on we use the abbreviation $\psi[t]=(\psi(t,\cdot),\partial_t \psi(t,\cdot))$.
\begin{definition}
Given initial data $(\psi_{0},\psi_{1})$, we define
\begin{align*}
T_{(\psi_{0},\psi_{1})} := \sup  \left\{
T >0 \middle|
  \begin{subarray}{c}  
  \exists \text{~solution~} \psi : C_{T} \longrightarrow \mathbb{R} \text{~to~} \eqref{cauchy} \text{~in the sense of}   \\ 
   \text{~Definition~} \ref{def} \text{~with initial data~} \psi[0]=(\psi_{0},\psi_{1}) |_{\mathbb{B}_{T}^{d}} 
  \end{subarray} \right\} \cup  \{0\}.
\end{align*}
In the case where $T_{(\psi_{0},\psi_{1})} < \infty$, we call $T \equiv T_{(\psi_{0},\psi_{1})}$ the blowup time at the origin.
\end{definition}
 We remark that the effective spatial dimension for the problem \eqref{cauchy} is $d+2$.
To see this, recall that, by regularity, we get the boundary condition $\eqref{reg}$. Therefore, it is natural to switch to the variable $\widehat \psi(t,r):=r^{-1}\psi (t,r)$.  Then \eqref{cauchy} transforms into
\begin{align*}
     \begin{cases}
        \widehat{\psi}  _{tt} (t,r) - \Delta ^{  \text{rad}  }_{r,d+2}  \widehat{\psi} (t,r)  = - \frac{d-1}{2} \frac{ \sin ( 2 r  \widehat{\psi}  (t,r) ) -2r \widehat{\psi} (t,r)}{r^3}, &\quad \text{in } C_{T} \\
        \widehat{\psi}  (0,r)= \frac{f(r)}{r},~~~  \widehat{\psi}  _{t} (0,r)= \frac{g(r)}{r}, &\quad \text{on } \{ t=0 \} \times [0,+\infty)  \ 
     \end{cases}
\end{align*}	
Note that the nonlinearity is now generated by a smooth function and the radial Laplacian is in $d+2$ dimensions. 
 
\begin{theorem} 
Fix $T_{0}>0$ and $d\geq 3$ odd. Then there exist constants $M,\delta,\epsilon >0$ such that for any radial initial data
$\psi [0]$  
satisfying 
\begin{align*}
\Big \| |\cdot|^{-1} \Big( \psi[0] -\psi^{T_{0}}[0] \Big) \Big \|_{ H^{\frac{d+3}{2}} (\mathbb{B}_{T_{0}+\delta}^{d+2}) \times H^{\frac{d+1}{2}} (\mathbb{B}_{T_{0}+\delta}^{d+2}  )} \leq \frac{\delta}{M}
\end{align*}
the following statements hold:
\begin{enumerate}
\item $T \equiv T_{\psi[0]} \in [T_{0}-\delta,T_{0}+\delta]$,
\item the solution $\psi :C_{T} \longrightarrow \mathbb{R}$ satisfies
\begin{align*}
(T-t)^{k-\frac{d}{2}} \Big \|  |\cdot|^{-1} \Big( \psi (t,\cdot) - \psi ^{T} (t, \cdot) \Big) \Big \|_{ \dot{H}^{k}(\mathbb{B}^{d+2}_{T-t} ) }  &\leq\delta (T-t)^{\epsilon} \\
(T-t)^{\ell+1-\frac{d}{2}} \Big \|  |\cdot|^{-1} \Big( \partial_t \psi (t,\cdot) - \partial_t \psi ^{T} (t, \cdot) \Big) \Big \|_{ \dot{H}^{\ell}(\mathbb{B}^{d+2}_{T-t} ) }  &\leq \delta (T-t)^{\epsilon}
\end{align*}
for all $k=0,1,2, \dots, \frac{d+3}{2}$ and $\ell=0,1,2\dots,\frac{d+1}{2}$.
\end{enumerate}
\end{theorem}

\begin{remark}
Note that the normalizing factors on the left-hand sides appear naturally and reflect the behavior of the self-similar solution $\psi^T$ in the respective homogeneous Sobolev norms,
i.e.,
\begin{align*}
  \||\cdot|^{-1}\psi^T(t,\cdot)\|_{\dot H^k(\mathbb B^{d+2}_{T-t})}&=\left \||\cdot|^{-1} f_0 \left(\frac{|\cdot|}{T-t} \right)  \right \|_{ \dot{H}^{k}( \mathbb{B}^{d+2}_{T-t} ) } =(T-t)^{\frac{d}{2}-k} \||\cdot|^{-1} f_0 \left(|\cdot| \right)  \|_{ \dot{H}^{k}( \mathbb{B}^{d+2}_{1} ) } 
  \end{align*}
  and
  \begin{align*}
\||\cdot|^{-1}\partial_t \psi^T(t,\cdot)\|_{\dot H^\ell(\mathbb B^{d+2}_{T-t})}&=(T-t)^{-2}\left \|f_0' \left(\frac{|\cdot|}{T-t} \right)  \right \|_{ \dot{H}^{\ell}( \mathbb{B}^{d+2}_{T-t} ) } \\
& =(T-t)^{\frac{d}{2}-\ell-1} \| f_0' \left(|\cdot| \right)  \|_{ \dot{H}^{\ell}( \mathbb{B}^{d+2}_{1} ) }  .
\end{align*}
\end{remark}

\subsection{Related results}
The question of singularity formation for the wave maps equation attracted a lot of interest in the recent past, in particular in the energy-critical case $d=2$. 
Bizo\'n-Chmaj-Tabor \cite{BizChmTab01} were the first to provide numerical evidence for the existence of blowup for critical wave maps with $\mathbb S^2$ target. 
Rigorous constructions of blowup solutions for this model are due to Krieger-Schlag-Tataru \cite{KriSchTat08}, Rodnianski-Sterbenz \cite{RodSte10}, and Rapha\"el-Rodnianski \cite{RapRod12}. Struwe \cite{Str03} showed that blowup for equivariant critical wave maps takes place via shrinking of a harmonic map.
This result was considerably generalized to the nonequivariant setting by Sterbenz-Tataru \cite{SteTat10a, SteTat10b}, see also Krieger-Schlag \cite{KriSch12} for a different approach to the large-data problem and e.g.~\cite{CotKenLawSch15a, CotKenLawSch15b, Cot15,
CanKri15, Sha16, LawOh16} for more recent results on blowup and large-data global existence. 

The energy-supercritical regime $d\geq 3$ is less understood. The small-data theory at minimal regularity is due to Shatah-Tahvildar-Zadeh \cite{ShaTah94} in the equivariant setting whereas Tataru \cite{Tat98, Tat01} and Tao \cite{Tao01a, Tao01b} treat the general case, see also \cite{KlaRod01, ShaStr02, NahSteUhl03, Kri03, Tat05}.
Self-similar blowup solutions were found by Shatah \cite{Sha88}, Turok-Spergel \cite{TurSpe90}, Cazenave-Shatah-Tahvildar-Zadeh \cite{CazShaTah98}, and Bizo\'n-Biernat \cite{BizBie15}.
The stability of self-similar blowup was investigated numerically in \cite{BizChmTab00, BizBie15, BieBizMal16} and proved rigorously in \cite{Don11, DonSchAic12, CosDonXia16, CosDonGlo16} in the case $d=3$. Furthermore, Dodson-Lawrie \cite{DodLaw15} proved that solutions with bounded critical norm scatter.

Finally, concerning the method, we remark that our proof relies on the techniques developed in the series of papers \cite{Don11, DonSchAic12, DonSch12, DonSch14, Don14, DonSch16, DonSch16a}.
However, we would like to emphasize that the present paper is not just a straightforward continuation of these works. In fact, new interesting issues arise, e.g.~in the spectral theory part, see Proposition \ref{projection} below. 

\section{Radial wave equation in similarity coordinates}
\label{sec:sim}

\noindent
To start our analysis, we rewrite the initial value problem $\eqref{cauchy}$ as an abstract Cauchy problem in a Hilbert space. First, we rescale the variable $\psi \equiv \psi (t,r)$ and switch to similarity coordinates. Then, we linearize around the rescaled blowup solution and derive the evolution problem satisfied by the perturbation. 
\subsection{Rescaled variables}
We define
\begin{align*}
\chi _{1} (t,r) := \frac{T-t}{r} \psi (t,r),\qquad \chi _{2} (t,r) := \frac{(T-t)^2}{r} \psi _{t} (t,r).
\end{align*}
Using the fact that $\psi$ is a solution to \eqref{cauchy}, we get
\begin{align*}
 \partial _{t} \chi _{1} (t,r) = &  - \frac{1}{T-t} \chi _{1} (t,r) +  \frac{1}{T-t} \chi _{2} (t,r), \\
 \partial _{t} \chi _{2} (t,r) =  & - \frac{2}{T-t} \chi _{2} (t,r) +  (T-t) \Delta ^{ \text{rad} }_{r,d} \chi _{1} (t,r) + \frac{2(T-t)}{r} \partial _{r} \chi _{1} (t,r) \\ 
& + (d-1) \frac{T-t}{r^2} \chi _{1} (t,r) - \frac{d-1}{2} (T-t)^2 \frac{ \sin \big( \frac{2r}{T-t} \chi _{1} (t,r) \big) }{r^3}.
\end{align*}
We introduce similarity coordinates
\begin{align*}
\mu: C_{T} \longrightarrow  \mathcal{C} ,~~ (t,r) \longmapsto \mu (t, r) =(\tau,\rho) := \Big( \log \Big ( \frac{T}{T-t} \Big ), \frac{r}{T-t} \Big ),
\end{align*}
which map the backward light-cone $C_T$ to the cylinder $\mathcal{C}:=(0,+ \infty) \times  [0,1]$. 
By the chain rule, the derivatives transform according to
\begin{align*}
\partial _{t} = \frac{e^{\tau}}{T} (\partial _{\tau} + \rho \partial _{\rho}),~ \partial _{r} = \frac{e^{\tau}}{T} \partial _{\rho},~ \partial ^{2}_{r} = \frac{e^{2 \tau}}{T^{2}} \partial ^{2}_{\rho},~ \Delta ^{ \text{rad} }_{r,d} = \frac{e^{2 \tau}}{T^{2}} \Delta ^{ \text{rad} }_{\rho,d}.
\end{align*}
Finally, setting
\begin{align*}
\psi _{j} (\tau, \rho) := \chi _{j} (t(\tau,\rho),r(\tau,\rho) ) = \chi _{j} (T(1-e^{-\tau}), T \rho e^{-\tau}), 
\end{align*}
for $j=1,2$, we obtain the system
\begin{align} \label{free}
& \begin{pmatrix}
   \partial _{\tau}  \psi _{1}  (\tau,\rho)  \\
   \partial _{\tau}  \psi _{2} (\tau,\rho) 
 \end{pmatrix} 
  =
  \begin{pmatrix}
   - \psi _{1}  (\tau,\rho) + \psi _{2} (\tau,\rho)   - \rho \partial _{\rho} \psi _{1}  (\tau,\rho)    \\
   \Delta ^{ \text{rad} }_{\rho,d+2} \psi _{1} (\tau,\rho) - \rho \partial _{\rho} \psi _{2} (\tau,\rho)  - 2 \psi _{2} (\tau,\rho)  
   \end{pmatrix}  \\ \nonumber 
& \quad  \quad   \quad   \quad  \quad   \quad  -  \frac{d-1}{ 2 \rho ^3}
   \begin{pmatrix}
   0 \\
   \sin ( 2 \rho \psi _{1} (\tau,\rho)  ) -2 \rho \psi _{1} (\tau,\rho)
   \end{pmatrix},
\end{align}
for $(\tau,\rho) \in \mathcal{C}$. 
 Note that the linear part is the free operator of the $(d+2)-$dimensional wave equation in similarity coordinates and the nonlinearity is perfectly smooth. Furthermore, the initial data transform according to
\begin{align*} 
\begin{pmatrix}
   \psi _{1}   (0,\rho) \\
   \psi _{2}   (0,\rho)
 \end{pmatrix} 
=   \frac{1}{\rho}  
\begin{pmatrix}
 f(T\rho) \\
  T g (T\rho)
 \end{pmatrix} 
=  \frac{1}{\rho}
\begin{pmatrix}
 \psi^{T_{0}}(0,T\rho) \\
 T\partial_0 \psi ^{T_{0}} (0,T\rho)
\end{pmatrix}+
\frac{1}{\rho}  
\begin{pmatrix}
F(T\rho) \\
T G(T\rho)
\end{pmatrix},
\end{align*}
for all $\rho \in [0,1]$. Here, $T_{0}>0$ is a fixed parameter and
\begin{align*}
& \psi ^{T_{0}} (0,T\rho) =  2 \arctan \left(  \frac{T }{T_{0}} \frac{\rho}{ \sqrt{d-2} }    \right),\quad \rho \equiv \rho (t,r):=\frac{r}{T-t}, \\ 
& F:=f-\psi ^{T_{0}}(0,\cdot), \quad G:=g-\partial_0 \psi^{T_{0}}(0,\cdot).
\end{align*}
We emphasize that the only trace of the parameter $T$ is in the initial data.

\subsection{Perturbations of the rescaled blowup solution} We linearize around the rescaled blowup solution and use the initial value problem for $(\psi_{1},\psi_{2})^{T}$ to obtain an initial value problem for the perturbation as an abstract Cauchy problem in a Hilbert space. For notational convenience we set
\begin{align*}
 \mathbf{ \Psi } (\tau)  (\rho) :=
 \begin{pmatrix}
   \psi _{1} (\tau, \rho)  \\
   \psi _{2} (\tau, \rho)
 \end{pmatrix} .
\end{align*}
The blowup solution is given by
\begin{align*} 
\mathbf{  \Psi } ^{ \text{res}  } (\tau) (\rho)
=
\begin{pmatrix}
   \frac{T-t}{r} \psi ^{T} (t,r)  \\
    \frac{(T-t)^2}{r} \psi _{t}^{T} (t,r) 
 \end{pmatrix} \Bigg | _{
 (t,r)=\mu ^{-1} (\tau,\rho)
} =
\begin{pmatrix}
   \frac{1}{\rho} f_0 (\rho) \\
   f_0'(\rho) 
 \end{pmatrix}, 
\end{align*} 
i.e., it is static.
We linearize around $\mathbf{\Psi}^{\text{res}}$ by inserting the ansatz $\mathbf{ \Psi }=  \mathbf{  \Psi } ^{ \text{res} } + \mathbf{  \Phi } $ into \eqref{free}.
For brevity we write
\begin{align*} 
\eta (x):= \sin(2x) - 2x,\quad x \in \mathbb{R}
\end{align*}
and use Taylor's theorem to expand the nonlinearity around 
$\frac{1}{\rho}f_0(\rho)$. We get
\begin{align*}
\sin \left(  2 \rho \psi _{1}  \right) - 2 \rho \psi _{1} & = \eta \left( \rho \psi _{1} \right) = \eta \left( f_0 + \rho \phi _{1} \right) = \eta \left( f_0  \right) + \eta ^{\prime} \left( f_0  \right) \rho \phi _{1} + N(\rho \phi _{1}),
\end{align*}
where, by definition, 
\begin{align*}
N (\rho \phi _{1} ) := \eta( f_0  + \rho \phi _{1}) - \eta ( f_0 ) - \eta ^{\prime} ( f_0 ) \rho \phi _{1}.
\end{align*}
We plug the ansatz and the Taylor expansion into Eq.~\eqref{free} which yields the abstract evolution equation 
\begin{align} \label{Evolution}
\left\{
	\begin{array}{ll}
	\partial _{\tau} \mathbf{ \Phi } (\tau) = \widetilde{\mathbf{  L } } \big( \mathbf{ \Phi } (\tau) \big) + \mathbf{ N } \big(  \mathbf{ \Phi }( \tau ) \big ), & \mbox{for } \tau \in  (0,+\infty)  \\
	\mathbf{ \Phi } (0)= \mathbf{U}(\mathbf{v},T),  
         \end{array}
\right.
\end{align}
for the perturbation
\begin{align*}
\mathbf{ \Phi } (\tau) (\rho)= 
\begin{pmatrix}
    \phi _{1} (\tau, \rho)   \\
  \phi _{2} (\tau, \rho) 
 \end{pmatrix}
 = 
\begin{pmatrix}
    \psi _{1} (\tau, \rho) -\frac{1}{\rho} f_0 (\rho)   \\
  \psi _{2} (\tau, \rho) - f_0' (\rho)
 \end{pmatrix}
\end{align*}
where
\begin{align}
&\widetilde{  \mathbf{  L } }  := \widetilde{ \mathbf{  L  } }_{0}  + \mathbf{  L ^{\prime} }, \label{1} \\
& \widetilde{ \mathbf{  L  } }_{0} \mathbf u (\rho):=
\begin{pmatrix}
- \rho u_1'(\rho)-u_1(\rho)+ u_2(\rho)  \\
  \Delta _{\rho,d+2}^{ \text{rad} } u_1(\rho) - \rho u_2'(\rho) - 2 u_{2}(\rho)   
 \end{pmatrix},  \label{2}  \\
&  \mathbf{  L ^{\prime} } \mathbf u(\rho):=
\begin{pmatrix}
   0 \\
   - \frac{d-1}{2} \frac{\eta' ( f_0(\rho) ) }{\rho ^2} u_{1}(\rho), 
 \end{pmatrix}, \label{3} \\
&  \mathbf{  N }(\mathbf u) (\rho):= 
 \begin{pmatrix}
   0 \\
   - \frac{d-1}{2} \frac{N( \rho u_{1}(\rho) )}{\rho ^3} 
 \end{pmatrix},  \label{4} 
\end{align}
for $\mathbf u=(u_1,u_2)$ and
\begin{align*}
 \eta'(f_0(\rho))=2  \cos (2f_0(\rho)) -2  = -16(d-2) \frac{\rho ^2}{ \left( \rho^2 +d-2 \right)^2  }. 
\end{align*}
Furthermore, the initial data are given by
\begin{align}
\label{5}
\mathbf \Phi(0)(\rho)=\mathbf U(\mathbf v,T)(\rho)=
\left (\begin{array}{c}
\frac{1}{\rho}f_0(\frac{T}{T_0}\rho) \\
\frac{T^2}{T_0^2}f_0'(\frac{T}{T_0}\rho) \end{array} \right )
-\left ( \begin{array}{c}
\frac{1}{\rho}f_0(\rho) \\ f_0'(\rho) \end{array} \right )
+\mathbf V(\mathbf v,T)(\rho)
\end{align} 
where
\[  \mathbf{ V } (\mathbf{v},T) (\rho):= 
 \begin{pmatrix}
   \frac{1}{\rho} F (T \rho)   \\
   \frac{T}{\rho} G (T \rho)  
 \end{pmatrix},~ \mathbf{v} := 
\begin{pmatrix}
   F   \\
   G 
 \end{pmatrix} .
 \]

\subsection{Strong light-cone solutions} 
To proceed, we need to define what it means to be a solution to the evolution problem $\eqref{Evolution}$. We introduce the Hilbert space
\begin{align*}
\mathcal{H} := H_{\text{rad}}^{\frac{d+3}{2}} (\mathbb{B}^{d+2} ) \times H_{\text{rad}}^{\frac{d+1}{2}} (\mathbb{B}^{d+2} ).
\end{align*}
Below we prove that the closure of the operator $\widetilde{\mathbf L}$, augmented with a suitable domain, generates a semigroup $\mathbf S(\tau)$ on $\mathcal H$.
This allows us to formulate $\eqref{Evolution}$ as an abstract integral equation via Duhamel's formula,
\begin{align} \label{Duhamel}
\Phi(\tau)= \mathbf{S}(\tau) \mathbf{U}(\mathbf{v},T)+ \int _{0}^{\tau} \mathbf{S}(\tau - s) \mathbf{N} \big( \Phi (s) \big)  ds.
\end{align}
Eq.~\eqref{Duhamel} yields a natural notion of strong solutions in light-cones.
\begin{definition} \label{def}
We say that $\psi:C_{T} \longrightarrow \mathbb{R}$ is a solution to $\eqref{cauchy}$ if the corresponding $\Phi:[0,\infty) \longrightarrow \mathcal{H}$ belongs to $C \big( [0,\infty);\mathcal{H} \big)$ and satisfies $\eqref{Duhamel}$ for all $\tau \ge 0$.
\end{definition}

\section{Proof of the theorem} 
\subsection{Notation} Throughout we denote by $\sigma(\mathbf{L}),~\sigma_{p}(\mathbf{L})$ and $\sigma_{e}(\mathbf{L})$ the spectrum, point spectrum, and essential spectrum, respectively, of a linear operator $\mathbf{L}$. Furthermore, we write $\mathbf{R}_{\mathbf{L}}(\lambda):=\left(\lambda - \mathbf{L} \right)^{-1}$, $\lambda \in \rho(\mathbf{L})$, for the resolvent operator where $\rho (\mathbf{L}):=\mathbb{C} \setminus \sigma(\mathbf{L})$ stands for the resolvent set. As usual, $a \lesssim b$ means $a \leq cb$ for an absolute, strictly positive constant $c$ which may change from line to line. Similarly, we write $a \simeq b$ if $a \lesssim b$ and $b \lesssim a$.
\subsection{Functional setting}
In the following we consider radial Sobolev functions $\hat{u} : \mathbb{B}_{R}^{d+2} \rightarrow \mathbb{C}$, that is, $\hat{u} (\xi)=u(|\xi|)$ for all $\xi \in \mathbb{B}_{R}^{d+2}$ where 
$u:(0,R) \rightarrow \mathbb{C}$. In particular, we define
\begin{align*}
u \in  
 H_{\text{rad}}^{m} ( \mathbb{B}_{R}^{d+2} ) \iff \hat{u} \in H^{m}  (\mathbb{B}_{R}^{d+2} ) := W^{m,2}  (\mathbb{B}_{R}^{d+2} ). 
 \end{align*}
The function space $H_{\text{rad}}^{m} (\mathbb{B}_R^{d+2} )$ becomes a Banach space endowed with the norm
\begin{align*}
\| u \|_{{H^{m}_{\text{rad}}} ( \mathbb{B}_{R}^{d+2} )} =  \| \hat{u}  \|_{ {H}^{m}  (\mathbb{B}_{R}^{d+2} )}.
\end{align*}
From now, we shall not distinguish between $u(|\cdot|)$ and $\hat{u}$. In addition, we introduce the Hilbert space 
\begin{align} \label{H}
\mathcal{H} := H_{\text{rad}}^{m} (\mathbb{B}^{d+2} ) \times H_{\text{rad}}^{m-1} (\mathbb{B}^{d+2} ),\quad m \equiv m_{d}:=\frac{d+3}{2}
\end{align}
associated with the induced norm
\begin{align*}
\| \mathbf{u} \|^{2}  =  \left \| (u_{1},u_{2})  \right \|^{2}  := \| u_{1} \|_{H_{\text{rad}}^{m} (\mathbb{B}^{d+2} )}^{2} + \| u_{2} \|_{H_{\text{rad}}^{m-1} (\mathbb{B}^{d+2} ) }^{2}.
\end{align*}
\subsection{Well-posedness of the linearized problem} 
We start with the study of the linearized problem and we convince ourselves that it is well-posed. Recall that the linear operator is given by $\eqref{1}$. To proceed, we follow \cite{DonSch16} and define the domain of the free part by
\begin{align*}
\mathcal{D} (\widetilde{ \mathbf{L} }_{0})  := \Big \{ \mathbf{u} \in  C^{\infty}(0,1)^2 \cap \mathcal H: w_{2} \in C^{2}\left([0,1]\right),~w_{1} \in C^{3} \left([0,1] \right),~ w_{1}^{\prime \prime} (0)=0 \Big \},
\end{align*}
where, for all $\rho \in [0,1]$ and $j=1,2$, 
\begin{align*}
w_{j} (\rho) := D_{d+2} u_{j} (\rho) := \Big (  \frac{1}{\rho} \frac{d}{d\rho}  \Big )^{ \frac{d-1}{2} } \big( \rho ^{d} u_{j}(\rho) \big) = \sum _{n=0}^{ \frac{d-1}{2} } c _{n} \rho^{n+1} u_{j}^{(n)} (\rho),
\end{align*}
for some strictly positive constants $c_{n}~ (n=0,1,\dots,\frac{d-1}{2})$. Note that the density of $C^{\infty} ( \overline{ \mathbb{B}^{d+2} } )$ in $H^{m}(\mathbb{B}^{d+2})$ implies the density of
\begin{align*}
\big( C_{\text{even}}^{\infty} [0,1] \big)^2 := \Big \{  \mathbf{u} \in \big( C^{\infty} [0,1] \big)^2:~~\mathbf{u}^{(2k+1)}(0)=0,~~k=0,1,2,\dots \Big \}  \subset \mathcal{D} ( \widetilde{ \mathbf{L} }_{0})
\end{align*}  
in $\mathcal{H}$ which in turn proves the density of $\mathcal{D} ( \widetilde{ \mathbf{L} }_{0})$ in $\mathcal{H}$. In other words, $\overline{ \mathcal{D} ( \widetilde{ \mathbf{L} }_{0}) } = \mathcal{H}$ and $\widetilde{ \mathbf{L} }_{0}$ is densely defined. 

\begin{prop} \label{growthestimatelinear}
The operator $\widetilde{ \mathbf{L} }_{0}: \mathcal{D}(\widetilde{ \mathbf{L} }_{0}) \subset \mathcal{H} \longrightarrow \mathcal{H} $ is closable and its closure $ \mathbf{L} _{0} : \mathcal{D}(\mathbf{L} _{0}) \subset \mathcal{H} \longrightarrow \mathcal{H} $ generates a strongly continuous one-parameter semigroup $( \mathbf{S}_{0}(\tau) )_{\tau \ge 0}$ of bounded operators on $\mathcal{H}$ satisfying the growth estimate
\begin{align} \label{growth}
\|  \mathbf{S} _{0} (\tau) \| \leq M e^{-\tau}
\end{align}
for all $\tau\geq 0$ and some constant $M\geq 1$.
In addition, the operator $\mathbf{L}:= \mathbf{L}_{0} + \mathbf{L}^{\prime}: \mathcal{D}(\mathbf{L}) \subset \mathcal{H} \longrightarrow \mathcal{H},~\mathcal{D}(\mathbf{L}) =\mathcal{D}(\mathbf{L}_{0})$, is the generator of a strongly continuous semigroup $( \mathbf{S}(\tau) )_{\tau \ge 0}$ on $\mathcal H$
and $\mathbf L': \mathcal H\to \mathcal H$ is compact.
\end{prop}
\begin{proof}
The fact that $\widetilde{ \mathbf{L} }_{0}$ is closable and its closure generates a semigroup satisfying the growth estimate $\eqref{growth}$ follows from Proposition 4.9 in \cite{DonSch16} by replacing $d$ in \cite{DonSch16} with $d+2$ and setting $p=3$.  It remains to apply the Bounded Perturbation Theorem to show that $\mathbf{L}:=\mathbf{L}_{0} + \mathbf{L}^{\prime}$ is the generator of a strongly continuous semigroup $( \mathbf{S}(\tau) )_{\tau \ge 0}$. In fact, we prove that $\mathbf{L}^{\prime}: \mathcal{H} \longrightarrow \mathcal{H}$, defined in $\eqref{3}$,  is compact. We pick an arbitrary sequence $(\mathbf{u}_{n})_{n \in \mathbb{N}} \subseteq \mathcal{H}$ that is uniformly bounded. By Lemma 4.2 in \cite{DonSch16}, $(D_{d+2} u_{1,n} )_{n \in \mathbb{N}}$ is uniformly bounded in $H^{2} (0,1)$ and the compactness of the Sobolev embedding $H^{2}(0,1) \xhookrightarrow{}  H^{1}(0,1)$ implies the existence of a subsequence, again denoted by $(D_{d+2} u_{1,n} )_{n \in \mathbb{N}}$, which is Cauchy in $H^{1}(0,1)$. Hence, for any $n,m \in \mathbb{N}$ sufficiently large, we get
\begin{align*}
\| \mathbf{L}^{\prime} \mathbf{u}_{n} - \mathbf{L}^{\prime}
  \mathbf{u}_{m} \| & \lesssim  \left \| \frac{ \eta' 
\circ f_0 }{|\cdot|^2} \right \|_{W^{1,\infty}(0,1)} \| D_{d+2} u_{1,n} - D_{d+2} u_{1,m} \|_{H^{1}(0,1)} \\
&  \simeq  \left \| \frac{  1}{ (|\cdot|^{2} + d-2)^2  }\right \|_{W^{1,\infty}(0,1)} \|D_{d+2} u_{1,n} - D_{d+2} u_{1,m} \|_{H^{1}(0,1)} \\
& \simeq   \| D_{d+2} u_{1,n} - D_{d+2} u_{1,m} \|_{H^{1}(0,1)},
\end{align*}
which shows that $(\mathbf{L}^{\prime} \mathbf{u}_{n})_{n \in \mathbb{N}}$ is Cauchy in $\mathcal{H}$. This proves that $\mathbf{L}^{\prime}$ is compact.
\end{proof}
\subsection{The spectrum of the free operator} \label{O1}
We can use the previous decay estimate for the semigroup $( \mathbf{S}_{0}(\tau) )_{\tau \ge 0}$ to locate the spectrum of the free operator $\mathbf L_0$. Indeed, by \cite{EngNag00}, p.~55, Theorem 1.10, we immediately infer
\begin{equation}
\label{spectrumLo}
 \sigma(\mathbf L_0)\subseteq \{\lambda\in \mathbb C: \text{Re}\lambda\leq -1\}. 
 \end{equation}

\subsection{The spectrum of the full linear operator} Next, we need to derive a suitable growth estimate for the semigroup $\mathbf{S}(\tau)$ and therefore turn our attention to the spectrum of the operator $\mathbf{L}$. To begin with, we consider the point spectrum.
\begin{prop} \label{OhaioProp}
We have
\begin{align*}
\sigma_p (\mathbf{L}) \subseteq \{ \lambda \in \mathbb{C}:~~\mathrm{Re} \lambda <0 \} \cup \{1\}.
\end{align*}
\end{prop}
\begin{proof}
We argue by contradiction and assume there exists a $\lambda \in \sigma_p(\mathbf{L}) \setminus \{ 1\}$ with $\mathrm{Re}\lambda \geq 0$. The latter means that there exists an element $\mathbf{u}=(u_{1},u_{2}) \in \mathcal{D} (\mathbf{L}) \setminus \{ 0 \}$ such that $\mathbf{u} \in$ $\ker(\lambda-\mathbf{L})$. A straightforward calculation shows that the spectral equation $(\lambda-\mathbf{L}) \mathbf{u} =0$ implies
\begin{small}
\begin{align*}
\big (1-\rho ^2 \big) u_{1}^{\prime \prime} (\rho) + \Bigg( \frac{d+1}{\rho} - 2(\lambda +2) \rho \Bigg)u_{1}^{\prime} (\rho)- \Bigg( (\lambda + 1) (\lambda +2)+ \frac{d-1}{2} V(\rho) \Bigg )  u_{1} (\rho) =0,
\end{align*} 
\end{small}
for $\rho \in (0,1)$, where 
\begin{align*}
V(\rho):=\frac{\eta'(f_0(\rho))}{\rho ^2} = \frac{-16(d-2)}{(\rho^2+d-2 )^2}.
\end{align*}
Since $\mathbf u\in \mathcal H$, we see that $u_{1}$ must lie in $H_{\text{rad}}^{\frac{d+3}{2}} (\mathbb{B}^{d+2})$. To proceed, we set $v_{1}(\rho):=\rho u_{1}(\rho)$. A straightforward computation implies that $v_{1}$ solves the second order ordinary differential equation
\begin{align} 
\label{eq:specode}
\big (1-\rho ^2 \big) v_{1}^{\prime \prime} (\rho) + \Bigg( \frac{d-1}{\rho} - 2(\lambda +1) \rho \Bigg)v_{1}^{\prime} (\rho)- \Bigg( \lambda  (\lambda +1)+ \frac{d-1}{2}\hat{V}(\rho) \Bigg )  v_{1} (\rho) =0,
\end{align}
for $\rho \in (0,1)$, where 
\begin{align*}
\hat{V}(\rho):= 2\frac{ \rho ^4 -6(d-2)\rho ^2 +(d-2)^2 }{\rho ^2 ( \rho^2 +d-2 )^2   }.
\end{align*}
We remark that this is the spectral equation studied in \cite{CosDonXia16, CosDonGlo16}.
Since all coefficients in \eqref{eq:specode} are smooth functions in $(0,1)$, we immediately get the a priori regularity $v_{1} \in C^{\infty}(0,1)$. 
We claim that $v_1\in C^\infty[0,1]$. To prove this, we employ Frobenius' method. The point $\rho =0$ is a regular singularity with Frobenius indices $s_{1}=1$ and $s_{2}=-(d-1)$. Therefore, by Frobenius theory, there exists a solution of the form
\begin{align*}
v^{1}_{1} (\rho)=\rho \sum _{i=0}^{\infty} x_{i} \rho ^{i} = \sum _{i=0}^{\infty} x_{i} \rho ^{i+1}, 
\end{align*} 
which is analytic locally around $\rho=0$. Moreover, since $s_{1}-s_{2}=d \in \mathbb{N}_{\text{odd}}$, there exists a second linearly independent solution of the form
\begin{align*}
v^{2}_{1} (\rho) &=C \log(\rho) v^{1}_{1} (\rho)  + \rho ^{-(d-1)} \sum _{i=0}^{\infty} y_{i} \rho ^{i}
\end{align*}
for some constant $C\in \mathbb C$ and $y_0=1$.
However, $v^{2}_{1}(\rho)/\rho$ does not lie in the Sobolev space $H_{\text{rad}}^{\frac{d+3}{2}} (\mathbb{B}^{d+2} )$ due to the strong singularity in the second term, no matter the value of the constant $C$. Consequently, $v_1$ must be a multiple of $v_1^1$ and we infer $v_1\in C^\infty[0,1)$. Similarly, the point $\rho =1$ is a regular singularity with Frobenius indices 
$s_{1}=0$ and $s_{2}=\frac{d-1}{2}-\lambda$.
Now we need to distinguish different cases. If $\frac{d-1}{2} -\lambda \notin \mathbb{Z}$, we have two linearly independent solutions of the form
\begin{align*}
& v_{1} ^{1}(\rho) =\sum_{i=0}^\infty x_i (1-\rho)^i, \\
& v_{1} ^{2}(\rho)=(1-\rho)^{\frac{d-1}{2}-\lambda}\sum_{i=0}^\infty y_i (1-\rho)^i
\end{align*} 
with $x_0=y_0=1$. The solution $v_{1} ^{2}(\rho)/\rho$ does not belong to the Sobolev space $H_{\text{rad}}^{\frac{d+3}{2}} (\mathbb{B}^{d+2} )$ and thus, $v_1\in C^\infty[0,1]$. In the case
 $\frac{d-1}{2} - \lambda:=k \in \mathbb{N}_{0}$, we have two fundamental solutions of the form
\begin{align*}
 v_{1} ^{1}(\rho)&=(1-\rho)^{k}\sum_{i=0}^\infty x_i(1-\rho)^i,\qquad x_0=1 \\
 v_{1} ^{2}(\rho)&= \sum_{i=0}^\infty y_i (1-\rho)^i+C\log (1-\rho)v_1^1(\rho),\qquad y_0=1
\end{align*} 
near $\rho= 1$. By assumption, $\mathrm{Re}\lambda\geq 0$ and thus, $k \leq \frac{d-1}{2}$. Hence, $v_1^2(\rho)/\rho$ does not lie in the Sobolev space $H_{\text{rad}}^{\frac{d+3}{2}} (\mathbb{B}^{d+2} )$ unless $C=0$ and we conclude $v_1\in C^\infty[0,1]$.  
Finally, if $\frac{d-1}{2}-\lambda=:-k$ is a negative integer, the fundamental system around $\rho=1$ has the form
\begin{align*} 
v_1^1(\rho)&=\sum_{i=0}^\infty x_i(1-\rho)^i \\
v_1^2(\rho)&=C\log(1-\rho)v_1^1(\rho)+(1-\rho)^{-k}\sum_{i=0}^\infty y_i (1-\rho)^i
\end{align*}
with $x_0=y_0=1$.
Again, $v_1^2(\rho)/\rho$ does not belong to $H_{\text{rad}}^{\frac{d+3}{2}} (\mathbb{B}^{d+2} )$ and we infer $v_1\in C^\infty[0,1]$ also in this case.
In summary, we have found a nontrivial solution $v_1\in C^\infty[0,1]$ to Eq.~\eqref{eq:specode} with $\mathrm{Re}\lambda \geq 0$, $\lambda\not=1$, but this contradicts \cite{CosDonXia16, CosDonGlo16}.
\end{proof}

The fact that $\mathbf L'$ is compact implies that the result on the point spectrum from Proposition \ref{OhaioProp} is already sufficient to obtain the same information on the full spectrum.

\begin{cor}
\label{cor:spec}
We have
 \begin{align*}
\sigma (\mathbf{L}) \subseteq \{ \lambda \in \mathbb{C}:~~\mathrm{Re} \lambda <0 \} \cup \{1\}.
\end{align*}
\end{cor}

\begin{proof}
Suppose there exists a $\lambda\in \sigma(\mathbf L)\setminus \{1\}$ with $\mathrm{Re}\lambda\geq 0$.
Then $\lambda\notin\sigma(\mathbf L_0)$ and thus, $\mathbf R_{\mathbf L_0}(\lambda)$ exists.
From the identity $\lambda-\mathbf L=[1-\mathbf L'\mathbf R_{\mathbf L_0}(\lambda)](\lambda-\mathbf L_0)$ we see that $1\in \sigma(\mathbf L'\mathbf R_{\mathbf L_0}(\lambda))$. Since $\mathbf L'\mathbf R_{\mathbf L_0}(\lambda)$ is compact, it follows that $1\in \sigma_p(\mathbf L'\mathbf R_{\mathbf L_0}(\lambda))$ and thus, there exists a nontrivial $\mathbf f \in \mathcal H$ such that $[1-\mathbf L'\mathbf R_{\mathbf L_0}(\lambda)]\mathbf f=0$.
Consequently, $\mathbf u:=\mathbf R_{\mathbf L_0}(\lambda)\mathbf f\not= 0$
satisfies $(\lambda-\mathbf L)\mathbf u=0$ and thus, $\lambda\in \sigma_p(\mathbf L)$.
This contradicts Proposition \ref{OhaioProp}.
\end{proof}

Next, we provide a uniform bound on the resolvent. 
To this end, we define
\begin{align*}
\Omega _{\epsilon,R} :=  \{  \lambda \in \mathbb{C}:~~~\text{Re}\lambda \geq -1+\epsilon, |\lambda|\geq R \}  
\end{align*}
for $\epsilon, R >0$.

\begin{prop} \label{O2}
Let $\epsilon >0$. Then there exist constants $R_{\epsilon}, C_{\epsilon}>0$ such that the resolvent $\mathbf{R}_{\mathbf{L}}$ exists on $\Omega _{\epsilon,R_\epsilon}$ and satisfies
\begin{align*}
\| \mathbf{R} _{\mathbf{L} } (\lambda)\| \leq C_{\epsilon}
\end{align*}
for all $\lambda\in \Omega_{\epsilon,R_\epsilon}$.
\end{prop}

\begin{proof}
Fix $\epsilon>0$ and take $\lambda\in \Omega_{\epsilon,R}$ for an arbitrary $R>0$. 
Then $\lambda\in \rho(\mathbf L_0)$ and the identity $(\lambda-\mathbf L)=[1-\mathbf L' \mathbf R_{\mathbf L_0}(\lambda)](\lambda-\mathbf L_0)$ shows that
$\mathbf R_{\mathbf L}(\lambda)$ exists if and only if
$1-\mathbf L'\mathbf R_{\mathbf L_0}(\lambda)$ is invertible. 
By a Neumann series argument this is the case if $\|\mathbf L'\mathbf R_{\mathbf L_0}(\lambda)\|<1$.

 To prove smallness of $\mathbf{L}^{\prime} \mathbf{R}_{\mathbf{L}_{0}} (\lambda)$, we recall the definition of $\mathbf L'$, Eq.~\eqref{3}, 
\begin{align*}
\mathbf{  L ^{\prime} } \mathbf{u} (\rho)=
\begin{pmatrix}
   0 \\
   - \frac{d-1}{2} V(\rho )u _{1} (\rho)
 \end{pmatrix}, 
\quad V(\rho)=\frac{\eta'(f_0(\rho))}{\rho ^2} = \frac{-16(d-2)}{(\rho^2+d-2 )^2}.
\end{align*}
Let $\mathbf u=\mathbf R_{\mathbf L_0}(\lambda)\mathbf f$ or, equivalently,
$(\lambda - \mathbf{L}_{0}) \mathbf{u} =\mathbf{f}$. The latter equation implies
\begin{align*}
(\lambda +1) u_{1} (\rho )  =u_{2} (\rho)- \rho u_{1}^{\prime} (\rho) + f_{1} (\rho).
\end{align*}
Now we use Lemma 4.1 from \cite{DonSch16} and $\|V^{(k)}\|_{L^\infty(0,1)}\lesssim 1$ for all $k\in \{0,1,\dots,m-1\}$ to obtain
\begin{align*}
 |\lambda +1| \| \mathbf{L}^{\prime} \mathbf{R}_{\mathbf{L} _{0}} (\lambda) \mathbf{f}  \| & = |\lambda +1|  \| \mathbf{L}^{\prime} \mathbf{u} \|  \simeq \big \| V \big( u_{2} - (\cdot) u_{1}^{\prime} +f_{1} \big) \big \|_{ H_{\text{rad}}^{m-1} (\mathbb{B}^{d+2} ) } \\
&  \lesssim \| u_{2} \|_{ H_{\text{rad}}^{m-1} (\mathbb{B}^{d+2} ) }  + \| (\cdot) u_{1}^{\prime} \|_{ H_{\text{rad}}^{m-1} (\mathbb{B}^{d+2} ) }  + \| f_{1} \|_{ H_{\text{rad}}^{m-1} (\mathbb{B}^{d+2} ) } \\
&  \lesssim \| u_{2} \|_{ H_{\text{rad}}^{m-1} (\mathbb{B}^{d+2} ) }  + \|  u_{1} \|_{ H_{\text{rad}}^{m} (\mathbb{B}^{d+2} ) }  + \| f_{1} \|_{ H_{\text{rad}}^{m-1} (\mathbb{B}^{d+2} ) } \\
& \simeq  \| \mathbf{u} \| + \| \mathbf{f} \|  \lesssim \Big( \frac{1}{\text{Re}\lambda +1}  + 1 \Big) \| \mathbf{f} \| \\
&\lesssim \|\mathbf f\|,
\end{align*}
where we have used the bound 
\[ \|\mathbf u\|=\|\mathbf R_{\mathbf L_0}(\lambda)\mathbf f\|\leq \frac{1}{\mathrm {Re}\lambda+1}\|\mathbf f\| \]
which follows from semigroup theory, see \cite{EngNag00}, p.~55, Theorem 1.10.
In other words,
\begin{align*}
\| \mathbf{L}^{\prime} \mathbf{R}_{\mathbf{L} _{0}} (\lambda)   \| \lesssim \frac{1}{ |\lambda +1| } \leq \frac{1}{|\lambda| -1} \leq \frac{1}{R-1}
\end{align*}
and by choosing $R$ sufficiently large, we can achieve the desired $\|\mathbf L'\mathbf R_{\mathbf L_0}(\lambda)\|<1$.
As a consequence, $[1-  \mathbf{L}^{\prime} \mathbf{R}_{\mathbf{L} _{0}} (\lambda) ]^{-1}$ 
exists and we obtain the bound 
\begin{align*}
\| \mathbf{R}_{\mathbf{L}} (\lambda) \| & = \| \mathbf{R}_{\mathbf{L}_{0}} (\lambda) [1-  \mathbf{L}^{\prime} \mathbf{R}_{\mathbf{L} _{0}} (\lambda) ]^{-1} \| \\
& \leq  \| \mathbf{R}_{\mathbf{L}_{0}} (\lambda) \| \| [1-  \mathbf{L}^{\prime} \mathbf{R}_{\mathbf{L} _{0}} (\lambda) ]^{-1} \| \\
& \leq  \| \mathbf{R}_{\mathbf{L}_{0}} (\lambda) \| \sum _{i=0}^{\infty} \| \mathbf{L}^{\prime} \mathbf{R}_{\mathbf{L} _{0}} (\lambda) \|^{i} \\
& \leq C_\epsilon.
\end{align*}
\end{proof}

\subsection{The eigenspace of the isolated eigenvalue} In this section, we convince ourselves that the eigenspace of the isolated eigenvalue $\lambda=1$ for the full linear operator $\mathbf{L}$ is spanned by 
\begin{align} \label{g} 
\mathbf{g} (\rho):=
\begin{pmatrix}
   g_{1} (\rho) \\
   g_{2} (\rho)
 \end{pmatrix}
 =
 \begin{pmatrix}
   \frac{1}{\rho ^2 +d-2} \\
    \frac{2(d-2)}{(\rho ^2 + d-2)^2}
 \end{pmatrix},
~~\rho \in [0,1].
 \end{align}
Consequently, we are looking for all $\mathbf{u}=(u_{1},u_{2}) \in \mathcal{D} (\mathbf{L}) \setminus \{ 0 \}$ such that $\mathbf{u} \in \ker(1-\mathbf{L})$. A straightforward calculation shows that the spectral equation $(1-\mathbf{L}) \mathbf{u} =0$ is equivalent to the following system of ordinary differential equations,
\begin{align*}
 \begin{cases}
      u_{2} (\rho)= \rho u_{1}^{\prime} (\rho)+ 2 u_{1} (\rho), & \\
      \big (1-\rho ^2 \big ) u_{1}^{\prime \prime} (\rho) + \Big( \frac{d+1}{\rho} - 6\rho \Big)u_{1}^{\prime} (\rho)- \Big( 6+ \frac{d-1}{2} \frac{\eta'(f_0(\rho))}{\rho ^2} \Big )  u_{1} (\rho) =0,  \ 
     \end{cases}
\end{align*} 
for $\rho \in (0,1)$. One can verify that a fundamental system of the second equation is given by
\begin{align*}
\Big  \{ \frac{1}{\rho ^2 + d-2},~~\frac{Q_{d-1}(\rho)}{\rho ^d (\rho ^2 + d-2)} \Big \}
\end{align*}
where $Q_{d-1}$ is a polynomial of degree $d-1$ with non-vanishing constant term. We can write the general solution for the second equation as
\begin{align*}
u_{1} (\rho) = C_{1}  \frac{1}{\rho ^2 + d-2} + C_{2} \frac{Q_{d-1}(\rho)}{\rho ^d (\rho ^2 + d-2)}. 
\end{align*}
We must ensure that $\mathbf{u} \in \mathcal{D}( \mathbf{L})$ which in particular implies that $u_{1}$ must lie in the Sobolev space $H_{\text{rad}}^{\frac{d+3}{2}} (\mathbb{B}^{d+2} )$. This requirement yields $C_{2}=0$ which in turn gives $\mathbf u \in \langle \mathbf{g} \rangle$. In conclusion, 
\begin{align} \label{ker}
\text{ker}(1-\mathbf{L}) = \langle \mathbf{g} \rangle,
\end{align}
as initially claimed. 
\subsection{Time evolution for the linearized problem}
We now focus on the time evolution for the linearized problem \eqref{Evolution}. Due to the presence of the eigenvalue $\lambda =1$, there exists a one dimensional subspace $\langle \mathbf{g} \rangle$ of initial data for which the solution grows exponentially in time. We call this subspace the unstable space. On the other hand, initial data from the stable subspace lead to solutions that decay exponentially in time. As we will show now, this time evolution estimates can be established using semigroup theory together with the previous results on the spectrum of the linear operators $\mathbf{L}_{0}$ and $\mathbf{L}$. To make this rigorous, we follow \cite{DonSch16} and use the fact that the unstable eigenvalue $\lambda=1$ is isolated to introduce a (non-orthogonal) projection $\mathbf{P}$. This projection decomposes the Hilbert space of initial data $\mathcal{H}$ into the stable and the unstable space. 
Most importantly, we must ensure that $\langle \mathbf{g}\rangle$ is the only unstable direction in $\mathcal{H}$. This is the key statement of the following proposition and it is equivalent to the fact that 
the algebraic multiplicity of the isolated eigenvalue $\lambda =1$, 
\begin{align*}
m_{a} (\lambda =1):=\mathrm{rank}\, \mathbf{P} =\dim\mathrm{rg}\, \mathbf{P},
\end{align*}
is equal to one. We denote by $\mathcal{B}(\mathcal{H})$ the set of bounded operators from $\mathcal{H}$ to itself and prove the following result.
\begin{prop} \label{projection}
There exists a projection 
\begin{align*}
\mathbf{P} \in \mathcal{B} (\mathcal{H}),\quad \mathbf{P}: \mathcal{H} \longrightarrow \langle \mathbf{g} \rangle,
\end{align*}
which commutes with the semigroup $\big( \mathbf{S}(\tau) \big) _{\tau \ge 0}$. In addition, we have
\begin{equation} 
\mathbf{S}(\tau) \mathbf{P} \mathbf{f} = e ^{\tau} \mathbf{P} \mathbf {f}, \label{semigroup1} 
\end{equation}
and there exist constants $C, \epsilon>0$ such that
\begin{equation}
 \|  (1-\mathbf{P}) \mathbf{S}(\tau) \mathbf{f} \| \leq C e^{-\epsilon \tau } \| (1- \mathbf{P}) \mathbf{f} \|,\label{semigroup2}
\end{equation}
for all $\mathbf {f} \in \mathcal{H}$ and $\tau \geq 0$.
\end{prop}

\begin{proof}
We argue along the lines of \cite{DonSch16}. Since the eigenvalue $\lambda =1$ is isolated, we can define the spectral projection
\begin{align*}
\mathbf{P}: \mathcal{H} \longrightarrow  \mathcal{H}, \quad \mathbf{P} := \frac{1}{2 \pi i} \int _{\gamma} \mathbf{R}_{\mathbf{L}} (\mu) d \mu, 
\end{align*}
where $\gamma : [0,2\pi] \longrightarrow \mathbb{C}$ is a positively orientated circle around $\lambda =1$ with radius so small that $\gamma \big( [0,2 \pi] \big) \subseteq \rho (\mathbf{L})$, see e.g.~\cite{Kat95}. The projection $\mathbf P$ commutes with the operator $\mathbf{L}$ and thus with the semigroup $\mathbf S(\tau)$. Moreover, $\mathbf{P}$ decomposes the Hilbert space as $\mathcal{H} =\mathcal M \oplus \mathcal N$, where $\mathcal M:=\rg \mathbf P$ and $\mathcal N:=\rg(1-\mathbf P)=\ker \mathbf{P}$. Most importantly, the operator $\mathbf{L}$ is decomposed accordingly into the parts $\mathbf{L}_{\mathcal M}$ and $\mathbf{L}_{\mathcal N}$ on $\mathcal M$ and $\mathcal N$, respectively. The spectra of these operators are given by
\begin{align} \label{spectrum}
\sigma \left( \mathbf L_{\mathcal N} \right ) = \sigma (\mathbf{L}) \setminus \{1\},\qquad \sigma \left( \mathbf{L}_{\mathcal M} \right ) = \{1\}.
\end{align}
We refer the reader to \cite{Kat95} for these standard results. 

To proceed, we break down the proof into the following steps: \\ \\
Step 1: We prove that $\rank\mathbf{P}:=\dim\rg\mathbf{P}<+\infty$. We argue by contradiction and assume that $\rank\mathbf{P}=+\infty$. Using \cite{Kat95}, p.~239, Theorem 5.28, the fact that $\mathbf{L}^{\prime}$ is compact (see Proposition $\ref{growthestimatelinear}$), and the fact that the essential spectrum is stable under compact perturbations (\cite{Kat95}, p.~244, Theorem 5.35), we obtain
\begin{align*}
\mathrm{rank}\,\mathbf{P} = +\infty  \Longrightarrow 1 \in \sigma _{e} (\mathbf{L}) = \sigma _{e} (\mathbf{L} -\mathbf{L}^{\prime})=\sigma _{e}(\mathbf{L}_{0}) \subseteq \sigma (\mathbf{L}_{0}).
\end{align*}
This contradicts \eqref{spectrumLo}. 
\\ \\
Step 2: We prove that $\langle \mathbf{g} \rangle=\mathrm{rg}\,\mathbf{P}$. It suffices to show $\mathrm{rg}\,\mathbf{P} \subseteq \langle \mathbf{g} \rangle$ since the reverse inclusion follows from the abstract theory. From Step 1, the operator $1-\mathbf{L}_{\mathcal M}$ acts on the finite-dimensional Hilbert space $\mathcal M=\rg \mathbf P$ and, from $\eqref{spectrum}$, $\lambda =0$ is its only spectral point. Hence, $1-\mathbf{L}_{\mathcal M}$ is nilpotent, i.e., there exists a $k\in \mathbb N$ such that
\begin{align*}
\big(  1-\mathbf{L}_{\mathcal M}  \big)^{k}  \mathbf{u}= 0
\end{align*}
for all $\mathbf{u} \in \mathrm{rg}\, \mathbf{P}$
and we assume $k$ to be minimal.
Recall $\eqref{ker}$ to see that the claim follows immediately if $k=1$. We proceed by contradiction and assume that $k\geq 2$. Then, there exists a nontrivial function $\mathbf{u} \in \rg \mathbf{P} \subseteq \mathcal{D}( \mathbf{L})$ such that 
$(1-\mathbf L_{\mathcal M})\mathbf u$ is nonzero and belongs to $\ker(1-\mathbf L_{\mathcal M})\subseteq \ker(1-\mathbf L)=\langle\mathbf g\rangle$.
This means that $\mathbf{ u } \in \rg\mathbf{P} \subseteq \mathcal{D} (\mathbf{L})$ satisfies $(1- \mathbf{ L }) \mathbf{ u } = \alpha \mathbf{ g }$, for some $\alpha \in \mathbb{ C }\setminus \{0\}$. Without loss of generality we set $\alpha=-1$ and a straightforward computation shows that the first component of $\mathbf u$ solves the second order differential equation
\begin{align*} 
\left(1-\rho ^2\right) u_{1}^{\prime \prime} (\rho) + \left( \frac{d+1}{\rho} -6 \rho \right ) u_{1} ^{\prime} (\rho) - \left ( 6 + \frac{d-1}{2} \frac{\eta'(f_0(\rho) ) }{\rho ^2} \right ) u_{1} (\rho) =  G(\rho),
\end{align*}
for $\rho \in (0,1)$, where 
\begin{align*}
G(\rho):= \frac{\rho ^2 +5(d-2)}{ (\rho ^2 + d-2)^2},~~\rho \in [0,1].
\end{align*}
In order to find the general solution to this equation, recall $\eqref{g}$ to see that 
\begin{align*}
\hat{u}_{1} (\rho):= g_{1} (\rho) = \frac{1}{\rho ^2 +d-2},~~~\rho \in (0,1)
\end{align*}
is a particular solution to the homogeneous equation
\begin{align*} 
\left(1-\rho ^2 \right) u_{1}^{\prime \prime} (\rho) + \left( \frac{d+1}{\rho} -6 \rho \right ) u_{1} ^{\prime} (\rho) - \left ( 6 + \frac{d-1}{2} \frac{\eta'(f_0(\rho) ) }{\rho ^2} \right ) u_{1} (\rho) = 0.
\end{align*}
To find another linearly independent solution, we use the Wronskian
\begin{align*}
\mathcal{W} (\rho) := (1-\rho ^2)^{\frac{d-5}{2} } \rho ^{-d-1}
\end{align*}
to obtain
\begin{align*}
\hat{u}_{2} (\rho) := \hat{u}_{1} (\rho) \int _{\rho _1}^{\rho}  (1- x ^2)^{\frac{d-5}{2}} x^{-d-1} (x^2 + d-2)^2 dx,
\end{align*}
for some constant $\rho_{1} \in (0,1)$ and for all $\rho \in (0,1)$. 
Note that we have the expansion
\[ \hat u_2(\rho)=\rho^{-d}\sum_{j=0}^\infty a_j\rho^j,\quad a_0\not= 0 \]
near $\rho=0$. Furthermore, if $d\geq 5$, $\hat u_2\in C^\infty(0,1]$ and we choose $\rho_1=1$ which yields the expansion
\[ \hat u_2(\rho)=(1-\rho)^{\frac{d-3}{2}}\sum_{j=0}^\infty b_j (1-\rho)^j,\qquad b_0\not= 0 \]
near $\rho=1$.
For $d=3$, we set $\rho_1=\frac12$ and the expansion of $\hat u_2$ near $\rho=1$ contains a term $\log(1-\rho)$.
We invoke the variation of constants formula to see that $u_1$ can be expressed as
\begin{align*}
u_{1} (\rho) & = c_{1}   \hat{u}_{1} (\rho) + c_{2} \hat{u}_{2} (\rho) \\
& + \hat{u}_{2} (\rho) \int _{0}^{\rho} \frac{ \hat{u}_{1}(y)G(y)y^{d+1} }{(1-y^2)^{ \frac{d-3}{2}  } } dy  - \hat{u}_{1} (\rho) \int _{0}^{\rho} \frac{ \hat{u}_{2}(y)G(y)y^{d+1} }{(1-y^2)^{ \frac{d-3}{2}  } } dy,
\end{align*}
for some constants $c_{1}, c_{2} \in \mathbb{C}$ and for all $\rho \in (0,1)$. 
The fact that $u_1\in H^{\frac{d+3}{2}}_\mathrm{rad}(\mathbb B^{d+2})$ implies
$c_2=0$ and we are left with
\begin{equation}
\label{eq:expru1}
u_{1} (\rho)  = c_{1}   \hat{u}_{1} (\rho) 
 + \hat{u}_{2} (\rho) \int _{0}^{\rho} \frac{ \hat{u}_{1}(y)G(y)y^{d+1} }{(1-y^2)^{ \frac{d-3}{2}  } } dy  - \hat{u}_{1} (\rho) \int _{0}^{\rho} \frac{ \hat{u}_{2}(y)G(y)y^{d+1} }{(1-y^2)^{ \frac{d-3}{2}  } } dy.
\end{equation}
If $d=3$, $\hat u_2(\rho)\simeq \log(1-\rho)$ near $\rho=1$ and thus, the last term in Eq.~\eqref{eq:expru1} stays bounded as $\rho\to 1-$ whereas the second term diverges unless 
\[ \int_0^1  \frac{ \hat{u}_{1}(y)G(y)y^{d+1} }{(1-y^2)^{ \frac{d-3}{2}  } } dy=0, \]
which, however, is impossible since the integrand is strictly positive on $(0,1)$. 
This contradicts $u_1\in H^{\frac{d+3}{2}}_\mathrm{rad}(\mathbb B^{d+2})$ and we arrive at the desired $k=1$.

Next, we focus on $d\geq 5$, where the last term in Eq.~\eqref{eq:expru1} is smooth on $[0,1]$. To analyze the second term, we set
\begin{align}
\label{def:Id}
\mathcal{I} _{d}(\rho) :=  \hat{u}_{2} (\rho) \int _{0}^{\rho} \frac{ F_{d}(y) }{(1-y)^{ \frac{d-3}{2}  } } dy, \quad \text{~~~~~}   ~ F_{d}(y):= \frac{ \hat{u}_{1}(y)G(y)y^{d+1} }{(1+y)^{\frac{d-3}{2}}}=\frac{y^{d+1}(y^2+5(d-2))}{(1+y)^{\frac{d-3}{2}}(y^2+d-2)^3}.
\end{align}
Note that $F_{5}(1)\not= 0$ and thus, the expansion of $\mathcal I_5(\rho)$ near $\rho=1$ contains a term of the form $(1-\rho)\log(1-\rho)$. 
Consequently, $\mathcal I_5''\notin L^2(\frac12,1)$ and this is a contradiction to $u_1\in H^4_{\mathrm{rad}}(\mathbb B^7)$. 
The general case is postponed to the appendix (Proposition \ref{prop:Id}) where it is shown that the function $\mathcal I_d$ is not analytic at $\rho=1$. This implies that the expansion of $\mathcal{I}_d(\rho)$ near $\rho =1$ contains a term $(1-\rho)^{\frac{d-3}{2}}\log(1-\rho)$ which again contradicts $u_1\in H^{\frac{d+3}{2}}_\mathrm{rad}(\mathbb B^{d+2})$.\\ \\
Step 3: Finally, we prove the estimates $\eqref{semigroup1}$ and $\eqref{semigroup2}$ for the semigroup. First, note that $\eqref{semigroup1}$ follows immediately from the facts that $\lambda =1$ is an eigenvalue of $\mathbf{L}$ with eigenfunction $\mathbf{g}$ and $\rg\mathbf{P}=\langle \mathbf{g} \rangle$. 
Furthermore, from Corollary \ref{cor:spec} and Proposition \ref{O2} we infer the existence of $C,\epsilon>0$ such that 
\[ \|\mathbf R_{\mathbf L}(\lambda)(1-\mathbf P)\|\leq C \]
for all $\lambda \in \mathbb C$ with $\mathrm{Re}\lambda\geq -2\epsilon$.
Consequently, the Gearhart-Pr\"uss Theorem, see \cite{EngNag00}, p.~302, Theorem 1.11, yields the bound \eqref{semigroup2}.
\end{proof}

\subsection{Estimates for the nonlinearity}
The aim of this section is to establish a Lipschitz-type estimate for the nonlinearity. Recall that the nonlinear term in $\eqref{Evolution}$ is given by
\begin{align*}
\mathbf{  N } ( \mathbf{u}  ) (\rho) = 
  \begin{pmatrix}
   0  \\
   \hat N (\rho, u_{1} (\rho))
 \end{pmatrix}
 :=
 \begin{pmatrix}
   0 \\
   - \frac{d-1}{2} \frac{N( \rho u _{1} (\rho) )}{\rho ^3}
 \end{pmatrix}.
\end{align*} 
To begin with, we claim that 
\begin{align*} 
&\hat N(\rho,u_1(\rho)) \\
&=4 (d-1) u_{1}^{2} (\rho) \int _{0}^{1}  \int _{0}^{1}\int _{0}^{1} \cos \left( 2z \left( f_0(\rho) +xy \rho u_{1} (\rho) \right) \right) \left( \frac{ f_0 (\rho)}{\rho} + xy u_{1} (\rho)\right) x dz dy dx. \\
\end{align*}
To see this, we use the fundamental theorem of calculus and the fact that $\eta ^{\prime \prime} (0)=0$ to write
\begin{align*} 
N( \rho u _{1} (\rho) ) & = \eta( f_0 (\rho) + \rho u _{1} (\rho) ) - \eta ( f_0 (\rho)) - \eta ^{\prime} ( f_0 (\rho)) \rho u _{1} (\rho) \\
&= \int _{f_0 (\rho)}^{f_0(\rho) + \rho u_{1} (\rho)} \eta ^{\prime} (s) ds - \eta ^{\prime} ( f_0 (\rho)) \rho u _{1} (\rho) \\
&=\rho u_{1} (\rho) \int _{0}^{1} \eta ^{\prime} (f_0 (\rho) + x \rho u_{1} (\rho)) dx - \eta ^{\prime} ( f_0 (\rho)) \rho u _{1} (\rho) \\
&=\rho u_{1} (\rho)  \int _{0}^{1}\left( \eta ^{\prime} (f_0 (\rho) + x \rho u_{1} (\rho)) - \eta ^{\prime} ( f_0(\rho)) \right) dx \\
&=\rho u_{1} (\rho)  \int _{0}^{1}\left( \int _{f_0(\rho)}^{f_0 (\rho) + x\rho u_{1} (\rho)} \eta ^{\prime \prime} (s) ds \right) dx \\
&=\rho^{2} u_{1}^{2} (\rho)  \int _{0}^{1}x  \int _{0}^{1} \eta ^{\prime \prime} (f_0 (\rho) + xy \rho u_{1} (\rho)) dy dx \\
&=\rho^{2} u_{1}^{2} (\rho)  \int _{0}^{1}x  \int _{0}^{1} \int _{0}^{f_0(\rho) + xy \rho u_{1}(\rho)} \eta ^{\prime \prime \prime } (s) ds dy dx \\
&=\rho^{2} u_{1}^{2} (\rho)  \int _{0}^{1}x  \int _{0}^{1} \int _{0}^{1} \eta ^{\prime \prime \prime } \left(  (f_0(\rho) +xy \rho u_{1} (\rho))z \right) \left( f_0 (\rho) + xy \rho u_{1} (\rho)\right)  dz dy dx \\
&=\rho^{3} u_{1}^{2} (\rho)  \int _{0}^{1} x \int _{0}^{1} \int _{0}^{1} \eta ^{\prime \prime \prime } \left(  (f_0(\rho) +xy \rho u_{1} (\rho))z \right) \left(\frac{ f_0 (\rho)}{\rho} + xy u_{1} (\rho)\right)  dz dy dx. \\
\end{align*}
For later purposes, we note that the function
\begin{align*} 
\hat N (\rho, \zeta)=4 (d-1)  \zeta ^{2}   \int _{0}^{1} \int _{0}^{1} \int _{0}^{1} \cos \left( 2z \left( f_0(\rho) +xy \rho  \zeta  \right) \right) \left( \frac{ f_0 (\rho)}{\rho} + xy  \zeta \right) x dz dy dx, 
\end{align*}
defined for all $(\rho, \zeta) \in [0,1] \times \mathbb{R},$ is perfectly smooth in both variables since
\begin{align*}
\frac{f_0 (\rho)}{\rho} =  \frac{2}{\rho} \arctan \left ( \frac{\rho}{\sqrt{d-2}}  \right) 
\end{align*}
is smooth at $\rho=0$.
Moreover, we define
\begin{align} \label{DefM} 
M (\rho, \zeta) := \partial _{\zeta} \hat N (\rho, \zeta) = 4 (d-1) \left (  A(\rho, \zeta) + B(\rho, \zeta) +C(\rho, \zeta) +D(\rho, \zeta)  \right),
\end{align}
where
\begin{align*}
& A(\rho,\zeta):= 2 \frac{f_0(\rho)}{\rho} \zeta  \int _{0}^{1} \int_{0}^{1}  \int _{0}^{1}  \cos  \left( 2z \left( f_0 (\rho) +xy \rho  \zeta \right) \right) x  dz dy dx, \\
& B(\rho, \zeta):= -2 f_0(\rho) \zeta ^{2}  \int_{0}^{1} \int_{0}^{1} \int_{0}^{1}  \sin \left( 2z \left(f_0 (\rho) + xy \rho \zeta \right) \right)  x^{2} y z dz dy dx, \\
& C(\rho, \zeta):=  3  \zeta ^{2}  \int_{0}^{1} \int_{0}^{1} \int_{0}^{1}  \cos \left( 2z \left( f_0(\rho) + xy \rho \zeta \right) \right)  x^{2} y  dz dy dx, \\
& D(\rho, \zeta):=  -2 \rho \zeta ^{3}  \int_{0}^{1} \int_{0}^{1} \int_{0}^{1}  \sin \left( 2z \left( f_0 (\rho) + xy \rho \zeta \right) \right)  x^{3} y^{2}  zdz dy dx.
\end{align*}
We denote by $\mathcal{B}_{\delta} \subseteq \mathcal{H}$ the ball of radius $\delta$ in $\mathcal{H}$ centered at zero, i.e.,
\begin{align*}
\mathcal{B}_{\delta}:= \left \{\mathbf{u} \in \mathcal{H}:~\left \| \mathbf{u} \right \|= \left \| (u_{1},u_{2}) \right \|_{ H_{\text{rad}}^{\frac{d+3}{2}} (\mathbb{B}^{d+2}) \times H_{\text{rad}}^{\frac{d+1}{2}} (\mathbb{B}^{d+2}) }
\leq \delta  \right \}. 
\end{align*}
The main result of this section is the following Lipschitz-type estimate.

\begin{lemma}
Let $\delta>0$.
Then we have
\begin{align} \label{Lipschitz}
\big \|  \mathbf{N (u)} -  \mathbf{N(v)}\big \|  \lesssim  (\|  \mathbf{u}  \|  +\|  \mathbf{v} \| ) \|  \mathbf{u}-\mathbf{v} \| 
\end{align}
for all $\mathbf{u}, \mathbf{v} \in \mathcal{B}_{\delta}$.
\end{lemma}

\begin{proof}
We start by fixing a $\delta>0$, we pick two elements $\mathbf{u}, \mathbf{v} \in \mathcal{B}_{\delta}$ and define the auxiliary function
\begin{align*}
\zeta (\sigma)(\rho)= \sigma u_{1}(\rho) + (1-\sigma) v_{1} (\rho),
\end{align*}
for $\rho \in (0,1)$ and $\sigma \in [0,1]$. The triangle inequality implies
\begin{align*}
\mathbf{u},\mathbf{v} \in \mathcal{B}_{\delta} \Longrightarrow \left \| u_{1} \right \|_{H_{\text{rad}}^{\frac{d+3}{2}} (\mathbb{B}^{d+2})}  \leq \delta ,~\left \| v_{1} \right \|_{H_{\text{rad}}^{\frac{d+3}{2}} (\mathbb{B}^{d+2})} \leq \delta  
\Longrightarrow  \left  \|  \zeta (\sigma) \right  \| _{ H_{\text{rad}}^{\frac{d+3}{2}} (\mathbb{B}^{d+2})  }  \leq  \delta,
\end{align*}
for all $\sigma \in [0,1]$. In other words, 
\begin{align*}
\zeta (\sigma) \in \mathscr{B}_{\delta}:= \left \{f \in H_{\text{rad}}^{\frac{d+3}{2}} (\mathbb{B}^{d+2})  :~\left \| f \right \|_{H_{\text{rad}}^{\frac{d+3}{2}} (\mathbb{B}^{d+2})  } \leq \delta  \right \},
\end{align*}
for all $\sigma \in [0,1]$. Now, we claim that to show $\eqref{Lipschitz}$, it suffices to establish the estimate
\begin{align} \label{M}
\left \| M(\cdot, f(\cdot) ) \right \|_{ H_{\text{rad}}^{\frac{d+3}{2}} (\mathbb{B}^{d+2})   } \lesssim \left \| f \right \|_{ H_{\text{rad}}^{\frac{d+3}{2}} (\mathbb{B}^{d+2})  }
\end{align}
for all $f \in \mathscr{B}_{\delta}$,
where $M$ is given by $\eqref{DefM}$. To see this, we use the algebra property
\begin{align*}
\| fg \|_{ H ^{\frac{d+3}{2}} (\mathbb{B}^{d+2}) } \lesssim \| f \|_{ H ^{\frac{d+3}{2}} (\mathbb{B}^{d+2}) } \| g \|_{ H ^{\frac{d+3}{2}} (\mathbb{B}^{d+2}) } ,
\end{align*}
which holds since $\frac{d+3}{2}>\frac{d+2}{2}$,
to estimate
\begin{align*} 
\big \| \mathbf{N(u)} -\mathbf{N(v)} \big\|  &= \big \|  \hat N (\cdot, u_{1}(\cdot) ) - \hat N (\cdot , v_{1}(\cdot)) \big \| _{H_{\text{rad}}^{\frac{d+1}{2}} (\mathbb{B}^{d+2}) } \\
&\leq
\big \|  \hat N (\cdot, u_{1}(\cdot) ) - \hat N (\cdot , v_{1}(\cdot)) \big \| _{H_{\text{rad}}^{\frac{d+3}{2}} (\mathbb{B}^{d+2}) }  \\
&=\left \|  \int _{v_{1}(\cdot)}^{u_{1}(\cdot)}  \partial _{2} \hat N (\cdot, \zeta )  d \zeta  \right \|  _{H_{\text{rad}}^{\frac{d+3}{2}} (\mathbb{B}^{d+2})}   \\
&=\left \|   \left( u_{1}(\cdot)  -  v_{1}(\cdot) \right) \int _{0}^{1} \partial _{2} \hat N (\cdot , \underbrace{  \sigma u_{1}(\cdot)+(1- \sigma) v_{1}(\cdot)}_{\zeta(\sigma) })  d \sigma \right \| _{H_{\text{rad}}^{\frac{d+3}{2}} (\mathbb{B}^{d+2}) }   \\
 \nonumber 
&\lesssim \left \|   u_{1} - v_{1} \right \| _{ H ^{\frac{d+3}{2}} (\mathbb{B}^{d+2}) }  \left \|  \int _{0}^{1} \partial _{2} \hat N (\cdot , \zeta (\sigma) ) d \sigma \right \| _{H_{\text{rad}}^{\frac{d+3}{2}} (\mathbb{B}^{d+2}) }  \\
& \lesssim \left \|   u_{1} - v_{1} \right \| _{ H ^{\frac{d+3}{2}} (\mathbb{B}^{d+2}) }   \int _{0}^{1}  \left \|  M (\cdot , \zeta (\sigma)(\cdot) ) \right \| _{H_{\text{rad}}^{\frac{d+3}{2}} (\mathbb{B}^{d+2}) }  d \sigma   \\
&  \lesssim  \left \| u_{1} - v_{1}  \right \| _{ H_{\text{rad}}^{\frac{d+3}{2}} (\mathbb{B}^{d+2}) }  \int _{0}^{1}  \left \|  \zeta (\sigma)  \right \| _{H_{\text{rad}}^{\frac{d+3}{2}} (\mathbb{B}^{d+2}) }   d \sigma  \\
& \lesssim \left \| u_{1} - v_{1}  \right \| _{ H_{\text{rad}}^{\frac{d+3}{2}} (\mathbb{B}^{d+2}) }  \int _{0}^{1} \left( \sigma  \left \|  u_{1} \right \| _{H_{\text{rad}}^{\frac{d+3}{2}} (\mathbb{B}^{d+2}) }  +(1-\sigma) \left \|  v_{1}    \right \| _{H_{\text{rad}}^{\frac{d+3}{2}} (\mathbb{B}^{d+2}) }  \right)  d \sigma   \\
& \lesssim \left \| u_{1} - v_{1}  \right \| _{ H_{\text{rad}}^{\frac{d+3}{2}} (\mathbb{B}^{d+2}) }  \left( \left \|   u_{1} \right \| _{H_{\text{rad}}^{\frac{d+3}{2}} (\mathbb{B}^{d+2}) }  + \left \|  v_{1}    \right \| _{H_{\text{rad}}^{\frac{d+3}{2}} (\mathbb{B}^{d+2}) }  \right) \\
&\lesssim   \left \| \mathbf{u} - \mathbf{v}  \right \| \left( \left \|   \mathbf{u} \right \|  + \left \|  \mathbf{v}   \right \| \right).
\end{align*}
It remains to prove $\eqref{M}$. To this end we use a simple extension argument (see e.g.~Lemmas B.1 and B.2 in \cite{DonSch16}) and Moser's inequality (\cite{Rau12}, p.~224, Theorem 6.4.1) to infer the existence of a smooth function $h: [0,\infty)\to [0,\infty)$ such that
\[ \|M(\cdot,f(\cdot))\|_{H^{\frac{d+3}{2}}_\mathrm{rad}(\mathbb B^{d+2})}\leq h\left (\|f\|_{L^\infty(\mathbb B^{d+2})}\right )\|f\|_{H^{\frac{d+3}{2}}_\mathrm{rad}(\mathbb B^{d+2})} \]
for all $f\in \mathscr B_\delta$.
By Sobolev embedding we have $\|f\|_{L^\infty(\mathbb B^{d+2})}\lesssim \|f\|_{H^{\frac{d+3}{2}}_\mathrm{rad}(\mathbb B^{d+2})}\leq \delta$
for all $f\in \mathscr B_\delta$ and \eqref{M} follows.
This concludes the proof.
\end{proof}

\subsection{The abstract nonlinear Cauchy problem} In this section, we focus on the existence and uniqueness of solutions to the Cauchy problem $\eqref{Evolution}$. In fact, by appealing to Definition $\ref{def}$, we consider the integral equation
\begin{align} \label{integralequation}
\Phi(\tau)= \mathbf{S}(\tau) \mathbf{u} + \int _{0}^{\tau} \mathbf{S}(\tau - s) \mathbf{N} \big( \Phi (s) \big)  ds, 
\end{align}
for all $\tau \ge 0$ and $\mathbf u\in \mathcal H$. We introduce the Banach space
\begin{align*}
\mathcal{X}:= \{ \Phi \in C( [0,\infty);\mathcal{H}) : ~~\| \Phi \|_{\mathcal{X}} := \sup _{\tau >0} e^{\epsilon \tau} \| \Phi(\tau) \| < + \infty  \}
\end{align*}
with $\epsilon>0$ from Proposition \ref{projection}. Moreover, we denote by $\mathcal{X}_{\delta}$ the closed ball
\begin{align*}
\mathcal{X}_{\delta} :=\left \{ \Phi \in \mathcal{X}: \| \Phi \|_{\mathcal{X}} \leq \delta \right \} = \left \{ \Phi \in C( [0,\infty);\mathcal{H}): \| \Phi \| \leq \delta e^{-\epsilon \tau},~~\forall \tau >0 \right \}.
\end{align*}

In the following, we will only sketch the rest of the proof and discuss the main arguments since they are analogous to \cite{Don11, DonSch12, DonSch14, Don14, DonSch16}. To prove the main theorem, we would like to apply a fixed point argument to the integral equation $\eqref{integralequation}$. However, the exponential growth of the solution operator on the unstable subspace prevents from doing this directly. We overcome this obstruction by subtracting the correction term\footnote{All integrals here exist as Riemann integrals over continuous functions.} 
\begin{align} \label{Correction}
\mathbf{C}(\Phi,\mathbf{u}) := \mathbf{P} \left(  \mathbf{u} + \int_{0}^{\infty} e^{-s} \mathbf{N} \big( \Phi (s) \big) ds \right)
\end{align}
from the initial data. Consequently, we consider the fixed point problem
\begin{align} \label{modified}
\Phi (\tau)= \mathbf{K} ( \Phi, \mathbf{u})(\tau)
\end{align}
where
\begin{align} \label{K}
\mathbf{K} (\Phi, \mathbf{u}) (\tau):=\mathbf{S} (\tau) [\mathbf{u} - \mathbf{C}(\Phi,\mathbf{u})] + \int_{0}^{\tau} \mathbf{S} (\tau - s) \mathbf{N} \big( \Phi(s) \big) ds. 
\end{align}
This modification stabilizes the evolution as the following result shows.

\begin{theorem} \label{th2}
There exist constants $\delta,C>0$ such that for every $\mathbf{u} \in \mathcal{H}$ with $\| \mathbf{u} \| \leq \frac{\delta}{C}$, there exists a unique $\mathbf{\Phi} (\mathbf{u}) \in \mathcal{X}_{\delta}$ that satisfies
\begin{align*}
\mathbf{\Phi} (\mathbf{u}) = \mathbf{K} (\mathbf{\Phi} (\mathbf{u}),\mathbf{u}).
\end{align*}
In addition, $\mathbf{\Phi}(\mathbf u)$ is unique in the whole space $\mathcal X$ and the solution map $\mathbf{u} \mapsto \mathbf{\Phi }(\mathbf{u})$ is Lipschitz continuous.
\end{theorem}
\begin{proof}
The proof is based on a fixed point argument and the essential ingredient is the Lipschitz estimate \eqref{Lipschitz} for the nonlinearity. Although the proof coincides with the one of Theorem 4.13 in \cite{DonSch16}, we sketch the main points for the sake of completeness. We pick $\delta>0$ sufficiently small and fix $\mathbf{u} \in \mathcal{H}$ with $\| \mathbf{u} \| \leq \frac{\delta}{C}$, where $C>0$ is sufficiently large. First, note that the continuity of the map
\begin{align*}
\mathbf{K} (\Phi, \mathbf{u}): [0,\infty) \longrightarrow \mathcal{H},\quad \tau \longmapsto \mathbf{K} (\Phi, \mathbf{u})(\tau)
\end{align*}
follows immediately from the strong continuity of the semigroup $\left( \mathbf{S}(\tau) \right)_{\tau >0}$. Next, to show that $\mathbf K(\cdot,\mathbf u)$ maps $\mathcal X_\delta$ to itself,
we pick an arbitrary $\Phi \in \mathcal{X}_{\delta} $ and decompose the operator according to
\begin{align*}
\mathbf{K} ( \Phi, \mathbf{u})(\tau) = \mathbf{P} \mathbf{K} ( \Phi, \mathbf{u})(\tau) + (1-\mathbf{P} ) \mathbf{K} ( \Phi, \mathbf{u})(\tau).
\end{align*}
The Lipschitz bound \eqref{Lipschitz} implies
\begin{align*}
\left \|  \mathbf{N} \left(  \Phi (\tau) \right) \right \| \lesssim \delta^2 e^{-2 \epsilon \tau}
\end{align*}
and together with the time evolution estimates for the semigroup on the unstable and stable subspaces (see Proposition \ref{projection}), we get
\begin{align*}
 \left \|  \mathbf{P} \mathbf{K} \left( \Phi, \mathbf{u} \right) (\tau) \right \| \lesssim \delta ^{2} e^{-2 \epsilon  \tau}, \quad \left \| \left(1- \mathbf{P} \right) \mathbf{K} \left( \Phi, \mathbf{u} \right) (\tau) \right \| \lesssim (\tfrac{\delta}{C}+\delta^2) e^{-\epsilon  \tau} .
\end{align*}
Clearly, these estimates imply that $ \mathbf{K} ( \Phi, \mathbf{u}) \in \mathcal{X}_{\delta}$ for sufficiently small $\delta$ and sufficiently large $C>0$. Finally, we need to show the contraction property. To this end, we pick two elements $\Phi,\widetilde{\Phi} \in \mathcal{X}_{\delta}$.
As before, the Lipschitz estimate \eqref{Lipschitz} together with Proposition \ref{projection} imply
\begin{align*}
 \left \| \mathbf{P} \left( \mathbf{K} ( \Phi, \mathbf{u})(\tau) - \mathbf{K} ( \widetilde{ \Phi }, \mathbf{u})(\tau)  \right) \right \| &\lesssim  \delta e^{-\epsilon \tau} \left \| \Phi - \widetilde{\Phi} \right \|_{\mathcal X}, \\
 \left \| \left(1- \mathbf{P}\right) \left( \mathbf{K} ( \Phi, \mathbf{u})(\tau) - \mathbf{K} ( \widetilde{ \Phi }, \mathbf{u})(\tau)  \right) \right \| &\lesssim  \delta e^{-\epsilon\tau} \left \| \Phi - \widetilde{\Phi} \right \|_{\mathcal X}
\end{align*}
and by choosing $\delta$ sufficiently small we conclude
\begin{align*}
\left \| \mathbf{K} ( \Phi, \mathbf{u}) - \mathbf{K} ( \widetilde{ \Phi }, \mathbf{u})  \right \|_{\mathcal{X}} \leq \frac{1}{2}  \left \| \Phi - \widetilde{\Phi} \right \|_{\mathcal{X}}.
\end{align*} 
Consequently, the claim follows by the contraction mapping principle. Uniqueness in the whole space $\mathcal X$ and the Lipschitz continuity of the solution map are routine and we omit the details.
\end{proof}

Now we turn to the particular initial data we prescribe. To this end, we define the space
\begin{align*}
\mathcal{H}^{R} := H_{\text{rad}}^{m} (\mathbb{B}_{R}^{d+2} ) \times H_{\text{rad}}^{m-1} (\mathbb{B}_{R}^{d+2}), \quad m\equiv m_{d} =\frac{d+3}{2}
\end{align*}
for $R>0$, endowed with the induced norm
\begin{align*}
\left \| \mathbf{w} \right \|_{\mathcal{H}^{R}}^{2} = \left \| (w_{1},w_{2}) \right \|_{\mathcal{H}^{R}}^{2}  =\left  \| w_{1}\right  \|_{ H_{\text{rad}}^{m} \left( \mathbb{B}^{d+2}_R \right) } + 
\left   \| w_{2} \right \|_{ H_{\text{rad}}^{m-1} \left( \mathbb{B}^{d+2}_R \right) }. 
\end{align*}
Recall the definition of the initial data operator $\mathbf U(\mathbf v, T)$ from Eq.~\eqref{5}.

\begin{lemma} \label{th1}
Fix $T_0>0$. 
Let $\delta>0$ be sufficiently small and $\mathbf{v}$ with $|\cdot|^{-1} \mathbf{v} \in \mathcal{H}^{T_{0} + \delta}$. Then, the map
\begin{align*}
\mathbf{U} (\mathbf{v},\cdot):[T_{0}-\delta,T_{0}+\delta] \longrightarrow \mathcal{H},\quad T \longmapsto \mathbf{U} (\mathbf{v},T)
\end{align*}
is continuous. Furthermore, for all $T \in [T_{0} -\delta,T_{0}+\delta]$,
\begin{align*}
\big \| |\cdot|^{-1} \mathbf{v}\big \|_{\mathcal{H}^{T_{0}+\delta} } \leq \delta \Longrightarrow  \big \| \mathbf{U} (\mathbf{v},T)\big  \|  \lesssim \delta.
\end{align*} 
\end{lemma}

\begin{proof}
The statements are straightforward consequences of the very definition of $\mathbf U(\mathbf v,T)$, the smoothness of $\frac{f_0(\rho)}{\rho}$, and the continuity of rescaling in Sobolev spaces. We omit the details.
\end{proof}

Finally, given $T_{0}>0$ and $\mathbf{v} \in \mathcal{H}^{T_{0} + \delta}$ with $\| |\cdot|^{-1} \mathbf{v} \|_{\mathcal{H}^{T_{0}+\delta} } \leq \frac{\delta}{M}$ for $\delta>0$ sufficiently small and $M>0$ sufficiently large, we apply Lemma \ref{th1} to see that $\mathbf{u}:=\mathbf{U} (\mathbf{v},T)$ satisfies the assumptions of Theorem \ref{th2} for all $T\in [T_0-\delta,T_0+\delta]$. Hence, for all $T \in [T_{0}-\delta,T_{0}+\delta]$, the map $\mathbf K(\cdot,\mathbf U(\mathbf v,T))$ has a fixed point $\Phi_{T}:= \mathbf{\Phi} (\mathbf U(\mathbf v,T)) \in \mathcal{X}_{\delta}$. In the last step we now argue that for each $\mathbf v$, there exists a particular $T_{\mathbf{v}} \in [T_{0}-\delta,T_{0}+\delta]$ that makes the correction term vanish, i.e., $\mathbf C(\Phi_{T_{\mathbf v}},\mathbf U(\mathbf v,T_{\mathbf v}))=0$. Since $\mathbf C$ has values in $\rg \mathbf P=\langle\mathbf g\rangle$, the latter is equivalent to 
\begin{align}
\label{eq:Tv}
\exists T_{\mathbf{v}} \in [T_{0}-\delta,T_{0}+\delta]: \quad \Big< \mathbf{C} \left( \Phi _{T_{\mathbf{v}}},\mathbf{U} \left( \mathbf{v},T_{\mathbf{v}} \right) \right),\mathbf{g} \Big>_{\mathcal H} = 0.
\end{align}
The key observation now is that
\[ \partial_T \left . \left ( \begin{array}{c}
\frac{1}{\rho}f_0(\frac{T}{T_0}\rho) \\
\frac{T^2}{T_0^2}f_0'(\frac{T}{T_0}\rho) \end{array} \right )\right |_{T=T_0}=\frac{2\sqrt{d-2}}{T_0}\,\mathbf g(\rho) \]
and thus, we have the expansion
\[ \Big< \mathbf{C} \left( \Phi_T,\mathbf U (\mathbf{v},T ) \right),\mathbf{g} \Big>_{\mathcal H}=\frac{2\sqrt{d-2}}{T_0}\|\mathbf g\|^2(T-T_0)
+O((T-T_0)^2)+O(\tfrac{\delta}{M}T^0)+O(\delta^2T^0).
\]
Consequently, a simple fixed point argument proves \eqref{eq:Tv}, see \cite{DonSch16}, Theorem 4.15 for full details.
In summary, we arrive at the following result.

\begin{theorem} \label{correction}
Fix $T_0>0$. Then there exist $\delta,M >0$ such that for any $\mathbf{v}$ with  
\[ \| |\cdot|^{-1} \mathbf{v} \|_{\mathcal{H}^{T_{0}+\delta} } \leq \frac{\delta}{M} \] there exists a $T \in [T_{0}-\delta,T_{0}+\delta]$ and a function $\Phi \in \mathcal{X}_{\delta}$ which satisfies 
\begin{align} 
 \label{eq:int}
\Phi (\tau) = \mathbf{S} (\tau) \mathbf{U} (\mathbf{v},T) + \int_{0}^{\tau} \mathbf{S} (\tau-s) \mathbf{N} \big( \Phi (s) \big) ds
\end{align}
for all $\tau \geq 0$. Furthermore, $\Phi$ is unique in $C \big( [0,\infty);\mathcal{H} \big)$.
\end{theorem}

\subsection{Proof of the main theorem} With the results of the previous section at hand, we can now prove the main theorem. Fix $T_0>0$ and suppose the radial initial data 
$\psi [0]$ satisfy
\begin{align*}
\left  \| |\cdot|^{-1} \Big( \psi [0] -\psi^{T_{0}}[0]  \Big) \right \|_{ H^{\frac{d+3}{2}} (\mathbb{B}_{T_{0}+\delta}^{d+2}) \times H^{\frac{d+1}{2}} (\mathbb{B}_{T_{0}+\delta}^{d+2}  )} \leq \frac{\delta}{M}
\end{align*}
with $\delta,M>0$ from Theorem \ref{correction}.
We set $\mathbf v:=\psi[0]-\psi^{T_0}[0]$, cf.~Section \ref{sec:sim}.
Then we have
\begin{align*}
\left \| |\cdot|^{-1} \mathbf{v} \right \|_{\mathcal{H}^{T_{0}+\delta}}  = \left \| |\cdot|^{-1} \Big( \psi [0] -\psi^{T_{0}}[0] \Big) 
\right \|_{\mathcal{H}^{T_{0}+\delta}} \leq \frac{\delta}{M}
\end{align*}
and Theorem $\ref{correction}$ yields the existence of 
$T \in [T_{0}-\delta,T_{0}+\delta]$ such that Eq.~\eqref{eq:int} has a unique solution $\Phi \in \mathcal{X}$ that satisfies $\| \Phi (\tau) \| \leq \delta e^{-\epsilon \tau}$ for all $\tau\geq 0$.
By construction,
\[ \psi(t,r)=\psi^T(t,r)+\frac{r}{T-t}\phi_1\left (\log\frac{T}{T-t},\frac{r}{T-t}\right ) \]
is a solution to the original wave maps problem \eqref{cauchy}. Furthermore, 
\[ \partial_t \psi(t,r)=\partial_t \psi^T(t,r)+\frac{r}{(T-t)^2}\phi_2 \left (\log\frac{T}{T-t},\frac{r}{T-t}\right ). \]
Consequently,
\begin{align*}
(T-t)^{k-\frac{d}{2}}&\left \||\cdot|^{-1}\left (\psi(t,\cdot)-\psi^T(t,\cdot)\right )
\right \|_{\dot H^k(\mathbb B^{d+2}_{T-t})}  \\
&=(T-t)^{k-\frac{d}{2}-1}\left \|\phi_1\left (\log\frac{T}{T-t},\frac{|\cdot|}{T-t}\right ) \right \|_{\dot H^k(\mathbb B^{d+2}_{T-t})} \\
&=\left \|\phi_1\left (\log\frac{T}{T-t},\cdot\right ) \right \|_{\dot H^k(\mathbb B^{d+2})}
\leq \left \|\Phi\left (\log\frac{T}{T-t}\right )\right \| \\
&\leq \delta (T-t)^\epsilon
\end{align*}
for all $t\in [0,T)$ and $k=0,1,2,\dots,\frac{d+3}{2}$.
Analogously,
\begin{align*}
(T-t)^{\ell-\frac{d}{2}+1}&\left \||\cdot|^{-1}\left (\partial_t \psi(t,\cdot)-\partial_t \psi^T(t,\cdot)\right )\right \|_{\dot H^\ell(\mathbb B^{d+2}_{T-t})} \\
&=(T-t)^{\ell-\frac{d}{2}-1}\left \|\phi_2\left (\log\frac{T}{T-t},\frac{|\cdot|}{T-t}\right ) \right \|_{\dot H^\ell(\mathbb B^{d+2}_{T-t})} \\
&=\left \|\phi_2\left (\log\frac{T}{T-t},\cdot\right ) \right \|_{\dot H^\ell(\mathbb B^{d+2})}
\leq \left \|\Phi\left (\log\frac{T}{T-t}\right )\right \| \\
&\leq \delta (T-t)^\epsilon
\end{align*}
for all $\ell=0,1,2,\dots,\frac{d+1}{2}$.

\appendix

\section{Properties of the function $\mathcal I_d$}

\noindent 
We first derive a consequence of results from \cite{CosDonGlo16} which then leads to the desired statement that $\mathcal I_d$ is not analytic at $1$. 
Recall the supersymmetric problem Eq.~(4.1) from \cite{CosDonGlo16},
\begin{equation}\label{eq:SUSY}
(1-\rho^2)\tilde{u}_\lambda''+\left[\frac{k+1}{\rho}-2(\lambda+1)\rho\right]\tilde{u}'_\lambda
-\lambda(\lambda+1)\tilde{u}_\lambda+\frac{2k}{\rho^2}\frac{\rho^2-k-2}{\rho^2+k}\tilde{u}_\lambda=0,
\end{equation} 
where $d=k+2$.

\begin{lemma}\label{lem:U}
Let $m\in \mathbb N$, $m\geq 2$, and $d=2m+1$.
Then the function 
	\begin{equation*}
\mathcal{U}_m(\rho):=(1-\rho^2)^{m-1}\int_{0}^{\rho}\frac{y^{2m+2}}{(1-y^2)^m}g_1(y)^2
dy,\qquad g_1(y)=\frac{1}{y^2+d-2}
	\end{equation*}
	is not analytic at $\rho=1$.
\end{lemma}

\begin{proof}
In view of the supersymmetric factorization derived in \cite{CosDonGlo16} (or by a direct computation) it follows that $\tilde u_1$ satisfies Eq.~\eqref{eq:SUSY} for $\lambda=1$ if and only if $\tilde v_1(\rho)=\rho^m (1-\rho^2)^{-\frac{m}{2}}\tilde u_1(\rho)$ satisfies 
\begin{equation}\label{Eq:Factored}
	(\partial_\rho-w(\rho))[(1-\rho^2)^2(\partial_\rho+w(\rho))]\tilde{v}_1(\rho)=0, 
\end{equation}
where $w=\frac{v_1'}{v_1}$ and
\[ v_1(\rho)=\rho^{m+1}(1-\rho^2)^{1-\frac{m}{2}} g_1(\rho). \]
Observe that the function $1/v_1$ solves Eq.~\eqref{Eq:Factored}.
Furthermore,
the Wronskian of two solutions of Eq.~\eqref{Eq:Factored} is of the form $\frac{c}{(1-\rho^2)^2}$ for some constant $c$ and thus,
the reduction formula yields another solution
\begin{align*}
\tilde{v}_1(\rho)&=\frac{1}{v_1(\rho)}\int_0^\rho \frac{v_1(y)^2}{(1-y^2)^2}dy 
=\frac{(1-\rho^2)^{\frac{m}{2}-1}}{g_1(\rho)\rho^{m+1}}\int_{0}^{\rho}\frac{y^{2m+2}}{(1-y^2)^m}g_1(y)^2dy.
\end{align*}
By construction,
\begin{equation*}
\tilde{u}_1(\rho)=\rho^{-m}(1-\rho^2)^\frac{m}{2}\tilde v_1(\rho)=\frac{(1-\rho^2)^{m-1}}{g_1(\rho)\rho^{2m+1}}\int_{0}^{\rho}\frac{y^{2m+2}}{(1-y^2)^m}g_1(y)^2dy=\frac{\mathcal U_m(\rho)}{g_1(\rho)\rho^{2m+1}}
\end{equation*}
is a solution to Eq.~\eqref{eq:SUSY}.
Clearly, $\tilde u_1$ is analytic at $\rho=0$. Suppose $\tilde u_1$ were analytic at $\rho=1$ also. Then we would have found a nontrivial solution $\tilde u_1\in C^\infty[0,1]$ to Eq.~\eqref{eq:SUSY} with $\lambda=1$. This, however, contradicts Theorem 4.1 in \cite{CosDonGlo16}. 
We conclude that $\tilde u_1$ and hence $\mathcal U_m$ must be nonanalytic at $\rho=1$.
\end{proof}

\begin{prop}
\label{prop:Id}
Let $d\geq 5$ be odd. 
Then the function $\mathcal I_d$ defined in Eq.~\eqref{def:Id} is not analytic at $\rho=1$.
\end{prop}

\begin{proof}
Since $\hat{u}_1=g_1$ and $G(y)=2yg'_1(y)+5g_1(y)$, we have
\[ \mathcal{I}_d(\rho)=\hat{u}_2(\rho)\int_{0}^{\rho}\frac{y^{d+1}}{(1-y^2)^{\frac{d-3}{2}}}\left [2yg_1(y)g'_1(y)+5g_1(y)^2 \right ] dy. \]
To simplify notation, we use the convention from above and write $d=2m+1$.
Since the order of the zero of $\hat{u}_2(\rho)$ at $\rho=1$ is $m-1$, it is enough to prove that
\[ \mathcal{J}_m(\rho):=(1-\rho^2)^{m-1}\int_{0}^{\rho}\frac{y^{2m+2}}{(1-y^2)^{m-1}}\left [2yg_1(y)g'_1(y)+5g_1(y)^2\right ]dy \]
is nonanalytic at $\rho=1$.
An integration by parts yields
\begin{align*}
	\mathcal{J}_m(\rho)&=(1-\rho^2)^{m-1}\int_{0}^{\rho}\frac{y^{2m-2}}{(1-y^2)^{m-1}}\frac{d}{dy}\left(y^5g_1(y)^2\right)dy\\
	&=\rho^{2m+3}g_1(\rho)^2-2(m-1)(1-\rho^2)^{m-1}\int_{0}^{\rho}\frac{y^{2m+2}}{(1-y^2)^{m}}g_1(y)^2dy
\end{align*}
and Lemma \ref{lem:U} completes the proof.
\end{proof}

\bibliographystyle{plain}
\bibliography{wmodd}

\end{document}